\documentclass[12pt]{article}
\usepackage[margin=2cm, top=2cm, bottom=2cm]{geometry}
\usepackage[utf8]{inputenc} 
\usepackage[T1]{fontenc}  
\usepackage{lmodern}
\usepackage{url}          
\usepackage{booktabs}      
\usepackage{amsfonts}     
\usepackage{nicefrac}      
\usepackage{microtype}    
\usepackage{mathrsfs}  

\usepackage{amsmath,amssymb,amsthm,xcolor}
\usepackage{array}
\usepackage{float}
\newtheorem{theorem}{Theorem}[section]

\newtheorem{remark}[theorem]{Remark}
\newtheorem{lemma}[theorem]{Lemma}

\newtheorem{proposition}[theorem]{Proposition}
\newtheorem{definition}[theorem]{Definition}
\newtheorem{assumption}[theorem]{Assumption}

\newcommand{\Halmos}{\mbox{\quad$\square$}}
\usepackage{graphicx}
\usepackage{subcaption}
\usepackage{mdframed}
\usepackage{lipsum}
\usepackage{wrapfig}
\usepackage[hidelinks,hypertexnames=false]{hyperref}
\usepackage{tikz}
\usepackage{enumitem}
\allowdisplaybreaks
\usepackage{epstopdf}
\usepackage{algorithm} 
\usepackage{algpseudocode} 
\usepackage{booktabs}
\usepackage{changepage}
\usepackage{enumitem}

\newcommand{\CC}{\mathbb{C}}
\newcommand{\EE}{\mathbb{E}}

\newcommand{\argmin}{\mathop{\rm argmin}}

\newcommand{\LCal}{\mathcal{L}}

\newcommand{\DCal}{\mathcal{D}}

\newcommand{\UCal}{\mathcal{U}}

\newcommand{\dom}{\textbf{dom}}
\newcommand{\interior}{\textbf{int}}

\newcommand{\br}{\mathbb{R}}

\newcommand{\ba}{\begin{array}}
\newcommand{\ea}{\end{array}}

\DeclareMathOperator{\diag}{diag}

\ifpdf
\DeclareGraphicsExtensions{.eps,.pdf,.png,.jpg}
\else
\DeclareGraphicsExtensions{.eps}
\fi

\title{Non-Convex Self-Concordant Functions: Practical Algorithms and Complexity Analysis}
\author{Donald Goldfarb\thanks{Department of Industrial Engineering and Operations Research, Columbia University (\url{goldfarb@columbia.edu}).}
\and Lexiao Lai\thanks{Department of Mathematics, The University of Hong Kong (\url{lai.lexiao@hku.hk}).} 
\and Tianyi Lin\thanks{Department of Industrial Engineering and Operations Research, Columbia University (\url{tl3335@columbia.edu}).} 
\and Jiayu Zhang\thanks{Department of Industrial Engineering and Operations Research, Columbia University (\url{jz3824@columbia.edu}).}
}
\date{}
\begin{document}

\maketitle
\vspace*{-5mm}%
\begin{center}
    \textbf{Abstract}
    \end{center}
    \vspace*{-4mm}
 \begin{adjustwidth}{0.2in}{0.2in}
~~~~We extend the standard notion of self-concordance to non-convex optimization and develop a family of second-order algorithms with global convergence guarantees. In particular, two function classes -- \textit{weakly self-concordant} functions and \textit{$F$-based self-concordant} functions -- generalize the self-concordant framework beyond convexity, without assuming the Lipschitz continuity of the gradient or Hessian. 
For
these function classes, we propose a regularized Newton method and an adaptive regularization method that achieve an $\epsilon$-approximate first-order stationary point in $O(\epsilon^{-2})$ iterations. Equipped with an oracle capable of detecting negative curvature, the adaptive algorithm can further attain convergence to an approximate second-order stationary point. Our experimental results demonstrate that the proposed methods offer superior robustness and computational efficiency compared to cubic regularization and trust-region approaches, underscoring the broad potential of self-concordant regularization for large-scale and neural network optimization problems.
\end{adjustwidth}



\section{Introduction}\label{sec:intro}
Newton's method is one of the cornerstones of optimization, dating back to the early development of iterative methods. In convex optimization, it is known for its local quadratic convergence when the objective function is smooth and strongly convex~\cite{Wright-2006-Numerical}.
However, Newton's method is not suitable in non-convex settings without appropriate safeguards, as it may converge to saddle points or exhibit
slower global convergence
than first-order methods~\cite{Cartis-2010-Complexity}. To address these limitations, extensive research has focused on enhancing its global behavior through various regularization techniques. Among them, cubic regularization methods~\cite{Griewank-1981-Modification, Nesterov-2006-Cubic, Nesterov-2008-Accelerating, Cartis-2011-Adaptive, Cartis-2011-Adaptive2} control 
the search direction size by penalizing large curvature through a cubic term; trust-region methods~\cite{Conn-2000-Trust, Curtis-2017-Trust, Curtis-2018-Concise, Curtis-2023-Worst} restrict each update to a region where the local quadratic model is accurate; and regularized
(also referred to as ``damped'') Newton's methods~\cite{Levenberg-1944-Method,Marquardt-1963-Algorithm,Tikhonov-1963-Solution,Polyak-2009-Regularized,Mishchenko-2023-Regularized, Doikov-2024-Gradient, Gratton-2025-Yet}
balance curvature exploitation and stability via adaptive damping terms. Collectively, these approaches have established rigorous global convergence guarantees and strong empirical performance across a broad range of non-convex optimization problems.

However, these approaches suffer from some drawbacks~\cite{Mishchenko-2023-Regularized, Gratton-2025-Yet}. A primary one is 
implementation: both cubic regularization methods~\cite{Nesterov-2006-Cubic} and trust-region methods~\cite{Conn-2000-Trust} require solving a nontrivial subproblem at each iteration that involves cubic or constrained quadratic models rather than
a simple Hessian inversion. While efficient strategies exist for these subproblems (see, e.g., \cite{Agarwal-2017-Finding, Dussault-2024-Scalable, Gould-1999-Solving, Adachi-2017-Solving}), 
they nevertheless introduce additional implementation complexity compared to a simple regularized Newton update, and their efficient deployment in high-dimensional settings remains an active area of research.
Moreover, the theoretical guarantees of cubic regularization methods and trust region methods depend on strong smoothness conditions, such as the Lipschitz continuity of the Hessian~\cite{Nesterov-2006-Cubic, Nesterov-2008-Accelerating, Polyak-2009-Regularized, Cartis-2011-Adaptive, Cartis-2011-Adaptive2, Conn-2000-Trust, Curtis-2017-Trust, Curtis-2023-Worst, Mishchenko-2023-Regularized, Doikov-2024-Gradient, Gratton-2025-Yet},
and sometimes even of the gradient~\cite{Curtis-2017-Trust, Curtis-2023-Worst, Gratton-2025-Yet}. These conditions are often unrealistic in various applications.

In the convex setting,
the notion of self-concordance~\cite{Nesterov-1994-Interior, Nesterov-2018-Lectures},
(see Definition \ref{def:SC-convex} below), 
provides a unified framework for studying second-order methods without Hessian Lipschitz continuity. More specifically, a convex and three times continuously differentiable function $f:\dom(f)\to\mathbb{R}$ is defined to be $\kappa$-self-concordant if for all $x\in \dom(f)$,
$\nabla^3 f(x)[h,h,h] \leq 2\kappa(\nabla^2f(x)[h,h])^{1.5}$. 
In particular, self-concordant functions may possess unbounded third-order derivatives, yet the growth of derivatives is locally controlled by the curvature encoded in the Hessian. This structure ensures that Newton steps remain well-behaved in regions where Hessian Lipschitz continuity fails. A canonical example is the logarithmic barrier function, which is self-concordant but lacks a Lipschitz continuous gradient or Hessian. For such functions, second-order methods (e.g.,
damped Newton) achieve global convergence and local quadratic convergence, offering both theoretical guarantees
and computational practicality.
Unfortunately, extending self-concordance to non-convex settings remains nontrivial since its definition
involves the square root of a local curvature term, which can become ill-defined when the Hessian has negative eigenvalues. One might attempt to define self-concordance as $\nabla^3 f(x)[h,h,h] \leq 2\kappa(|\nabla^2f(x)[h,h]|)^{1.5}$ when $f$ is potentially non-convex, but this modification imposes restrictive curvature bounds that fail to capture simple non-convex structures; for example, the sigmoid function $x \mapsto 1/(1+e^{-x})$ violates this property. This gap motivated
our development of (i)
a generalized notion of self-concordance suitable for non-convex objective
functions that preserves
its curvature control while accommodating regions of negative curvature,
and (ii) the development and analysis of second-order algorithms that exploit this property.
\paragraph{Related work on self-concordance} The classical notion of self-concordance was introduced by~\cite{Nesterov-1994-Interior} as a fundamental tool for analyzing interior-point methods in convex optimization \cite{Ye-1994-nl, Nesterov-1997-Self, Guler-1997-Hyperbolic}. Since then, this notion has been
extended
to composite optimization problems~\cite{Tran-2015-Composite}, distributed Newton methods~\cite{Zhang-2015-Disco}, randomized Newton updates~\cite{Lu-2017-Randomized} and stochastic quasi-Newton methods~\cite{Gao-2019-Quasi}. In parallel, other work has studied stronger variants of this notion~\cite{Rodomanov-2021-Greedy,Hanzely-2022-Damped}, which impose tighter curvature control to improve numerical stability.
Another line of work has adopted the notion of quasi-self-concordance, which encompasses important
objectives such as logistic regression~\cite{Bach-2010-Self, Tran-2015-Composite, Bach-2014-Adaptivity}.
Finally, in~\cite{Sun-2019-Generalized},
the notion of generalized self-concordance was introduced. Building upon this,
Frank-Wolfe methods have been developed that demonstrate the versatility of self-concordance-based curvature control beyond the standard convex setting~\cite{Dvurechensky-2020-Self, Carderera-2021-Simple, Dvurechensky-2023-Generalized}.
We note that all of the above extensions of self-concordance apply \textbf{only to convex functions.}

\paragraph{Related work on regularized methods for non-convex optimization}
Numerous regularization techniques have been developed to ensure global convergence of Newton-type methods that are applied to minimizing non-convex functions.
Such strategies include cubic regularization~\cite{Nesterov-2006-Cubic, Nesterov-2008-Accelerating}, trust-region
~\cite{Conn-2000-Trust}, and
recent regularized Newton variants~\cite{Mishchenko-2023-Regularized, Doikov-2024-Gradient, Zhou-2025-Regularized, Gratton-2025-Yet}, all designed to stabilize Newton steps in the presence of non-convex curvature. A major milestone in this line of work is the adaptive cubic regularization method~\cite{Cartis-2011-Adaptive, Cartis-2011-Adaptive2}, which established a global iteration complexity of $O(\epsilon^{-1.5})$ for minimizing functions with Lipschitz continuous Hessians. This framework has inspired extensive research aimed at improving regularization strategies, reducing computational cost, and refining complexity guarantees~\cite{Curtis-2017-Trust, Curtis-2018-Concise, Curtis-2021-Trust, Curtis-2023-Worst, Dussault-2024-Scalable, Gratton-2025-Yet, Zhang-2025-Homogeneous}. Adaptive cubic regularization has also been extended to constrained non-convex settings, including problems with convex constraints~\cite{Cartis-2012-Adaptive} and those involving conic constraints equipped with logarithmically homogeneous self-concordant barriers~\cite{Dvurechensky-2025-Hessian}. To enhance computational efficiency, several works have exploited Hessian-vector products as major oracles, achieving convergence rates on the order of $O(\epsilon^{-7/4})$ while avoiding explicit Hessian computations~\cite{Agarwal-2017-Finding, Carmon-2018-Accelerated}. A complementary line of research~\cite{Mccormick-1977-Modification, More-1979-Use, Goldfarb-1980-Use} has leveraged negative curvature directions to escape saddle points and approach approximate second-order stationary points. Building on these ideas,~\cite{Royer-2018-Complexity, Royer-2020-Newton, He-2025-Newton} proposed hybrid line-search algorithms that alternate between Newton updates and negative-curvature steps identified through the Lanczos method~\cite{Kuczynski-1992-Estimating}, offering a robust trade-off between curvature exploitation and global progress.

\paragraph{Contributions} This paper extends the notion of self-concordance to non-convex optimization and develops corresponding second-order methods with theoretical guarantees and empirical validation. As in convex settings, our framework accommodates functions with unbounded higher-order derivatives, thereby covering a broad range of applications. Our contributions can be summarized as follows:
\begin{enumerate}
\item We introduce two function classes -- weakly self-concordant and $F$-based self-concordant functions -- that extend the classical notion of self-concordance to non-convex settings. These classes preserve key curvature
properties
while encompassing many objectives arising in machine learning.
\item We propose regularized Newton and adaptive regularization methods
that make use of our weakly and $F$-based self-concordant notions
and prove their global and local convergence
without
requiring
gradient or Hessian Lipschitz continuity, and
any line-search. We extend the adaptive regularization methods~\cite{Cartis-2011-Adaptive, Cartis-2011-Adaptive2, Cartis-2022-Evaluation} to constrained problems with self-concordant barriers, where the regularization parameter is tuned automatically.
\item We conduct experiments on non-negative matrix factorization and neural network training, demonstrating that our methods
achieve superior robustness and efficiency compared with existing second-order approaches.
\end{enumerate}

\paragraph{Organization} In Section~\ref{sec:prelims}, we define non-convex self-concordant functions, and provide three examples of their applicability. In Section~\ref{sec:algorithms}, we describe our methods and present convergence results. 
In Section~\ref{sec:exp}, we present empirical results complementing our theoretical results, 
and offer some conclusions in Section~\ref{sec:conclu}. 
In an Appendix
we present some additional results on weakly self-concordant functions, including an inexact proximal point method for minimizing them.

\section{Non-Convex Weakly and $F$-based Self-Concordant Functions}\label{sec:prelims}
Letting $f:\br^n\rightarrow \br \cup \{+\infty\}$ be three-times continuously differentiable on $\dom(f) = \{x \in \br^n \mid f(x) < +\infty\}$, we consider the problem  
\begin{equation}\label{prob:unconstrained}
\min_{x\in \br^n} f(x).
\end{equation}
\begin{definition}~{\rm \cite[Definition 5.1.1]{Nesterov-2018-Lectures}}
\label{def:SC-convex}
Let $\kappa>0$. A closed and convex function $f: \br^n \to \br \cup \{+\infty\}$ is $\kappa$-self-concordant if $\dom(f)$ is open, $f$ is three times continuously differentiable on $\dom(f)$, and $\nabla^3 f(x)[h,h,h] \le 2\kappa \big(\nabla^2 f(x)[h,h]\big)^{3/2}$ for all $x \in \dom(f)$ and $h \in \br^n$.
\end{definition}

We now extend this notion to non-convex settings.
\begin{definition}\label{def:SC-nonconvex-1}
A function $f: \mathbb{R}^n \to \mathbb{R} \cup \{+\infty\}$ is $(\kappa,\ell)$-weakly self-concordant if the function $g(x) = f(x) + \frac{\ell}{2}\|x\|^2$ is $\kappa$-self-concordant and satisfies $\nabla^2g(x)\succ 0$ for all $x\in\dom(f)$.
\end{definition}
Weakly self-concordant functions form a subclass of weakly convex functions~\cite{Davis-2019-Stochastic}. Indeed, $f$ is $(\kappa,\ell)$-weakly self-concordant if and only if $f(x) + \frac{\ell}{2}\|x\|^2$ is closed,
convex, 
with a Hessian that is non-degenerate,
and
\begin{equation*}
\nabla^3 f(x)[h,h,h] = \nabla^3\left(f + \tfrac{\ell}{2}\|\cdot\|^2\right)(x)[h,h,h] \leq 2\kappa \left(\nabla^2 f(x)[h,h] + \ell\|h\|^2\right)^{3/2},
\end{equation*}
for all $x \in \dom(f)$ and $h \in \br^n$.
In particular, if $f$ is $\ell$-weakly convex and satisfies $|\nabla^3 f(x)[h,h,h]| \leq m\|h\|^3$ for all $x\in\dom(f)$, then $f$ is $(\frac{m}{2a^{3/2}}, \ell + a)$-weakly self-concordant for any $a > 0$.

The requirement of $\nabla^2g(x)\succ 0$ is technical and can be circumvented by choosing a larger $\ell$, assuming that $f$ is weakly convex. In fact, when implementing variants of Newton's method for convex self-concordant functions, the objective is also assumed to be strictly convex~\cite{Nesterov-2018-Lectures}.

\begin{definition}\label{def:SC-nonconvex-2}
Given a base function $F: \mathbb{R}^n \to \mathbb{R} \cup \{+\infty\}$, we say that $f$ is \emph{$F$-based $\kappa$-self-concordant} if $\dom(f) \subseteq \dom(F)$ and $g = f + F$ is $\kappa$-self-concordant and satisfies $\nabla^2g(x)\succ 0$ for all $x\in\dom(f)$.
\end{definition}

Again, the requirement of $\nabla^2g(x)\succ 0$ is technical and can be circumvented by adding a quadratic term to $F$. Classical strictly convex self-concordant functions are $0$-based self-concordant. 
Moreover, $(\kappa,\ell)$-weakly self-concordance is equivalent to $\frac{\ell}{2}\|\cdot\|^2$-based $\kappa$-self-concordance.
This formulation of $F$-based self-concordance unifies analysis across different settings and enjoys convenient closure properties.
\begin{proposition}\label{prop:SC-closure}
Let $\alpha_1,\alpha_2 > 0$. If $f_i$ is $F_i$-based $\kappa_i$-self-concordant for $i=1,2$, then $\alpha_1f_1+\alpha_2f_2$ is $(\alpha_1F_1+\alpha_2F_2)$-based $\max(\frac{\kappa_1}{\sqrt{\alpha_1}}, \frac{\kappa_2}{\sqrt{\alpha_2}})$-self-concordant.
\end{proposition}
\begin{proof}{Proof}
Since
$f_i+F_i$ is $\kappa_i$-self-concordant for $i=1,2$, 
~\cite[Theorem 5.1.1]{Nesterov-2018-Lectures} shows that $\alpha_1(f_1+F_1) + \alpha_2(f_2+F_2)$ is $\max(\frac{\kappa_1}{\sqrt{\alpha_1}}, \frac{\kappa_2}{\sqrt{\alpha_2}})$-self-concordant. This implies the desired result. \Halmos
\end{proof}

\subsection{Examples}
We now present three representative examples of $F$-based self-concordant functions with full details.
\paragraph{Generalized Phase Retrieval (GPR)}
As in~\cite{Sun-2018-Geometric},
the GPR loss function $f: \br^{2n} \to \br$
is defined as
\begin{equation*}
f(z^R, z^I) = \tfrac{1}{2m}\sum_{k=1}^m (|a_k^* z|^2-y_k^2)^2, \quad \textnormal{ where }  a_k = \tfrac{1}{\sqrt2}(a_k^R +a_k^I\sqrt{-1})\in \CC^n, \ z = z^R+z^I\sqrt{-1},
\end{equation*}
and
$a_k^*$ represents the conjugate transpose of $a_k$.
$f(z^R, z^I)$ is a $(\kappa,\ell)$-weakly self-concordant function for some $\kappa>0$ and $\ell>0$. Indeed, we can write $f(z^R, z^I)$ as
\begin{equation*}
f(z^R, z^I) = f(x) =\tfrac{1}{2m}\sum_{k=1}^m
\left(\tfrac{1}{2}
((c(1)_k^\top x)^2+(c(2)_k^\top x)^2) -y_k^2\right)^2,
\end{equation*}
where $x^\top = ((z^R)^\top,(z^I)^\top)$,
$c(1)_k^\top = ((a_k^R)^\top, (a_k^I)^\top)$ and $c(2)_k^\top =(-(a_k^I)^\top, (a_k^R)^\top)$.

By Proposition \ref{prop:SC-closure}, to prove that $f(z^R, z^I) = f(x)$ is $F$-based or weakly self-concordant we only need to show this for each term in the above summation, and in particular for $g(x) = ((\bar{x}_1)^2+(\bar{x}_2)^2)^2$, where for the $k$-th term  $\bar{x_i} \equiv (c(i)_k^\top x)$ for
$i = 1,2$.
Defining
$\bar{h_i} \equiv (c(i)_k^\top h)$ for
$i = 1,2$, we have that
\begin{equation*}
\nabla^2 g(x)[h,h] = 8(\bar x_1 \bar h_1+\bar x_2 \bar h_2)^2 + 4(\bar x_1^2+\bar x_2^2)(\bar h_1^2+ \bar h_2^2),
\end{equation*}
and for $\ell \geq \|c(1)_k\|^2+\|c(2)_k\|^2$,
that
\begin{equation*}
(\nabla^2g(x)[h,h]+\ell\|h\|^2)^{\frac{3}{2}} \geq (\bar{h}_1^2+\bar{h}_2^2)^{3/2}\left(4(\bar{x}_1^2+\bar{x}_2^2)+1\right)^{3/2},
\end{equation*}
\begin{equation*}
\nabla^3 g(x) [h,h,h]  = 24(\bar x_1 \bar h_1 + \bar x_2 \bar h_2)(\bar h_1^2+\bar h_2^2) \leq 24\sqrt{\bar{x}_1^2+\bar{x}_2^2}\;(\bar{h}_1^2+\bar{h}_2^2)^{3/2}.
\end{equation*}
Hence,
$g$ is $(3, \ell)$-weakly self-concordant.

\paragraph{Polynomial Optimization Problems (POPs)}
The above results for GPR can be generalized; indeed, any polynomial function is $F$-based self-concordant,
and hence, also
{\it ``$F$-based convex'',} as shown below in Proposition~\ref{prop:SC-poly},
for some function $F$. However, many polynomial functions (e.g., $-x^4 + 3x^2$) are not weakly convex and hence not weakly self-concordant. Therefore, it is necessary to incorporate the concept of $F$-based self-concordance into our analysis.

\begin{proposition}\label{prop:SC-poly}
Let $p \geq 2$ be an integer. Suppose that $f: \br^n \to \br$ is a multivariate polynomial with degree $\deg(f) \leq 2p$. Then, $F(x)=(\|x\|^2+1)^p$ is convex and there exists $m \geq 0$ such that $f$ is $mF$-based $1$-self-concordant.
\end{proposition}
\begin{proof}{Proof}
Since $\deg(f) \leq 2p$ and $\deg(F) = 2p$, each term of $\nabla^2 f(x)$, $\nabla^3 f(x)$, and  $\nabla^3 F(x)$, is respectively, a polynomial of degree at most $2p-2$, $2p-3$, and exactly $2p-3$. There exist $c_1,c_2\geq 0$ and $c_3 > 0$ such that, for all $x,h \in \br^n$, we have $|\nabla^2 f(x)[h,h]| \leq c_1(1+\|x\|^2)^{p-1}\|h\|^2$,
$|\nabla^3 f(x)[h,h,h]| \leq c_2(1+\|x\|^2)^{p-3/2}\|h\|^3$, 
and
$|\nabla^3 F(x)[h,h,h]| \leq c_3(1+\|x\|^2)^{p-\frac{3}{2}}\|h\|^3$.

Defining
$\varphi(r)=(r+1)^p$,
we have $\nabla F(x)=2\varphi'(\|x\|^2)x$, $\nabla^2 F(x)=2\varphi'(\|x\|^2)I+4\varphi''(\|x\|^2)xx^\top$, and for all $h \in \br^n$,
\begin{equation*}
\nabla^2 F(x)[h,h] = 2\varphi'(\|x\|^2)\|h\|^2 + 4\varphi''(\|x\|^2)(x^\top h)^2 \geq 2p(\|x\|^2+1)^{p-1}\|h\|^2.
\end{equation*}
Thus, $F$ is convex,
$\nabla^2\left(f+mF\right)(x)[h,h] \geq (2pm-c_1)(\|x\|^2+1)^{p-1}\|h\|^2$
and $|\nabla^3 \left( f+mF\right)(x)[h,h,h]| \leq (c_2+m c_3)(\|x\|^2+1)^{p-3/2}\|h\|^3$, which show that
there exists $m \geq 0$ such that $f+mF$ is $1$-self-concordant. \Halmos
\end{proof}
We note that there are many
constrained 
POPs, $\min_{x \in C} f(x)$, where $f$ is a polynomial, and
there exists a self-concordant barrier function $g: \br^n \to \br$ such that $\dom(g) = \interior(C)$. By Proposition~\ref{prop:SC-poly}, we have $f+F$ is self-concordant on $\br^n$ for some $F:\br^n\to \br$. Thus, we have $f + F + g$ is self-concordant on $\interior(C)$ and $f\mid_{\interior(C)}$
is $(F+g)$-based self-concordant. This perspective provides numerous examples of $F$-based self-concordant functions, such as
matrix factorization $\min_{U, V} \|M - UV^\top\|_F^2$
and sensor network localization~\cite{Tang-2024-Riemannian} $\min_{x_1, \ldots, x_n\in\br^d} \sum_{(i,j) \in \mathcal{E}} \bigl(\|x_i - x_j\|^2 - d_{ij}^2\bigr)^2$, where $\mathcal{E}\subseteq \{(i,j):i<j,i,j\in \{1,...,n\}\}$ and $d_{ij}\geq 0$ are constants.

Many problems
other than POPs involve loss functions that can be modified to be $F$-based self-concordant for some function $F$.
Our next example SDL is such a problem. We define $h_\alpha(x) = \alpha^{-1}(\log(1+e^{\alpha x})+\log(1+e^{-\alpha x}))$. For large $\alpha>0$, $h_\alpha$ serves as a smooth approximation to $|x|$ \cite{Schmidt-2007-Fast}, and it is weakly self-concordant. After some calculation, it can be verified that $\frac{\partial^2}{\partial x^2} h_\alpha(x) = \frac{\alpha}{2}(1-\tanh^2(\frac{\alpha x}{2}))$, $\frac{\partial^3}{\partial x^3} h_\alpha(x) = -\frac{\alpha^2\tanh(\alpha x/2)}{2\cosh^2(\alpha x/2)}$, and that both of $\left|\frac{\partial^2}{\partial x^2}h_\alpha(x)\right|$ and $\left|\frac{\partial^3}{\partial x^3}h_\alpha(x)\right|$ are bounded.

\paragraph{Sparse Dictionary Learning (SDL)} As introduced in \cite{Mairal-2008-Supervised}, we consider the optimization problem
\begin{equation*}
\min_{\|d_i\|\leq 1, i \in N } f(D, r_1, \ldots, r_K) = \sum_{i=1}^K \|x_i - D r_i\|^2 + \lambda\left(\sum_{i=1}^K \sum_{j=1}^n |r_i^{(j)}|\right),
\end{equation*}
where $N = \{ 1,\ldots,n\}$, $x_i \in \br^m$, $r_i \in \br^n$, $D \in \br^{m \times n}$, $d_i$ denotes the $i$-th column of $D$, $r_i^{(j)}$ denotes the $j$-th entry of $r_i$, and $\lambda>0$. We modify $f$ in two ways; specifically, we
apply $h_\alpha$ to smoothly approximate $|r_i^{(j)}|$
and introduce
the self-concordant function $g(d_i)=-\log(1-\|d_i\|^2)$ (see \cite[Theorem 5.1.4]{Nesterov-2018-Lectures}) to impose $\|d_i\| <1$. This yields the following approximation
\begin{equation*}
f_{\alpha, \mu}(D, r_1, \ldots, r_K) = \sum_{i=1}^K \|x_i - D r_i\|^2 + \lambda\left(\sum_{i=1}^K \sum_{j=1}^n h_\alpha(r_i^{(j)})\right) + \mu \left(\sum_{i=1}^n g(d_i)\right),
\end{equation*}
to $f$.
Propositions~\ref{prop:SC-closure} and \ref{prop:SC-poly} imply that $f_{\alpha,\mu}$ is $F$-based $1$-self-concordant.

\paragraph{Nonnegative Matrix Factorization (NMF)} Let $\br_+ = \{x \in \br: x \geq 0\}$ and $\br_{++} = \{x \in \br: x > 0\}$.
We consider here two alternative formulations of the problem of approximating a nonnegative matrix $Z$ by the product $XY$ of two nonnegative matrices $X$ and $Y$ \cite{Lee-2000-Algorithms}. The first formulation measures the difference $D$ between $Z$ and $XY$ by the square of the Frobenius norm of $D$ (i.e., the mean squared error (MSE)), while the second formulation measures $D$ by the
Kullback–Leibler (KL) divergence between $Z$ and $XY$; see~\eqref{prob:NMF_1} and 
\eqref{prob:NMF_2}
below.
Both fit into our non-convex self-concordant framework. Indeed, the first formulation yields the optimization problem
\begin{equation}\label{prob:NMF_1}
\min_{X \in \br_{++}^{m \times r}, Y \in \br_{++}^{r \times n}} f_{\text{MSE}}(X,Y) = \tfrac{1}{2mn}\|Z-XY\|_F^2 = \tfrac{1}{2mn} \left(\sum_{i=1}^{m}\sum_{j=1}^n \left(Z_{ij} - \sum_{k=1}^r X_{ik} Y_{kj}\right)^2\right),
\end{equation}
where $Z \in \br^{m \times n}_+$. By Proposition \ref{prop:SC-poly}, $f_\text{MSE}$ is $\ell F_\text{MSE}$-based $\kappa$-self-concordant for $F_\text{MSE}(X,Y)=(\|X\|_F^2+\|Y\|_F^2+1)^2 - \sum_{i, k} \log(X_{ik}) - \sum_{k, j} \log(Y_{kj})$ and some $\ell,\kappa>0$. 
The second formulation yields the problem
\begin{equation}\label{prob:NMF_2}
\min_{X \in \br_{++}^{m \times r}, Y \in \br_{++}^{r \times n}} f_{\text{KL}}(X, Y) = \tfrac{1}{mn} \left(\sum_{i=1}^m \sum_{j=1}^n D\left(Z_{ij} \Big\Vert \sum_{k=1}^r X_{ik} Y_{kj} \right)\right),
\end{equation}
where the KL divergence between $x$ and $y$ is defined as $D(x \Vert y) = x\log(x/y) - x + y$.

To prove that $f_{\text{KL}}(X,Y)$ is $\ell F_{\text{KL}}$-based $\kappa$-self-concordant for some convex function $F_{\text{KL}}$ and for some $\ell, \kappa>0$ on $\br_{++}^{m\times r} \times \br_{++}^{r \times n}$,
let $g(x,y)=-\log(x^\top y)$ and $G(x,y)=-\sum_{i=1}^r \log(x_iy_i)$. By Proposition \ref{prop:SC-closure}, it suffices to prove that $g$ is $\tau G$-based $1$-self-concordant for some $\tau>3$. Fixing $(x,y) \in \br_{++}^r \times \br_{++}^r$ and $h = (h_x, h_y) \in \br^r \times \br^r$, we define
\begin{equation*}
s=x^\top y, \quad L=h_x^\top y + h_y^\top x, \quad u_i=\tfrac{|(h_x)_i|}{x_i}, \quad v_i=\tfrac{|(h_y)_i|}{y_i}, \quad w_i=\tfrac{x_i y_i}{s}, \quad S=\textstyle\sum_{i=1}^r(u_i^2+v_i^2).
\end{equation*}
Then, we have
\begin{equation*}
\tfrac{|L|}{s} \leq \textstyle\sum_{i=1}^r w_i(|u_i|+|v_i|) \leq \sqrt{2\left(\textstyle\sum_{i=1}^r w_i(u_i^2+v_i^2)\right)}
\leq \sqrt{2S}, \quad \tfrac{|h_x^\top h_y|}{s} \leq \textstyle\sum_{i=1}^r \tfrac{|(h_x)_i||(h_y)_i|}{x_i y_i} \leq \tfrac{1}{2}S.
\end{equation*}
It follows that
\begin{equation*}
|\nabla^2 g(x,y)[h,h]| = \left|\tfrac{L^2}{s^2}-\tfrac{2h_x^\top h_y}{s}\right| \leq 3S,\quad |\nabla^3 g(x,y)[h,h,h]| = \left|-2\tfrac{L^3}{s^3}+\tfrac{6Lh_x^\top h_y}{s^2}\right| \leq 7\sqrt{2}S^{3/2}.
\end{equation*}
In addition, we obtain that $\nabla^2 G(x,y)[h,h]=S$ and $|\nabla^3 G(x,y)[h,h,h]| \leq 2(\sum_{i=1}^r (u_i^3+v_i^3)) \leq 2S^{3/2}$. Thus, we have
\begin{equation*}
\nabla^2(g+\tau G)[h,h]\geq(\tau-3) S, \quad \nabla^3(g+\tau G)[h,h,h] \leq  (2\tau+7\sqrt{2})S^{3/2}.
\end{equation*}
This implies that $|\nabla^3(g+\tau G)[h,h,h]| \leq 2(\nabla^2(g+\tau G)[h,h])^{3/2}$ for some $\tau>3$. This yields the desired result, i.e.,
 $F_{\text{KL}}(X,Y)= - \frac{1}{mn} (\sum_{i, k} \log(X_{ik}) + \sum_{k, j} \log(Y_{kj}))$ and $\ell > 3$.

\subsection{Descent 
Guarantees for $F$-based $\kappa$-self-concordant Functions}
For $F$-based $\kappa$-self-concordant functions, where $F$ is a convex function, we now establish
a lower bound on the decrease that is achievable by an appropriate choice of the step length along a direction of descent.
For simplicity, we define two auxiliary functions $\omega,\omega_\star:[0,\infty)\to\br\cup\{+\infty\}$ as follows:
\begin{equation*}
\omega(z) = z-\log(1+z) \quad
\textnormal{ and }
\quad\omega_\star(z) = \begin{cases} -z-\log(1-z) & \textnormal{if } z < 1,\\ +\infty & \textnormal{otherwise}. \end{cases}
\end{equation*}
This notation is standard in convex self-concordant analysis \cite{Nesterov-2018-Lectures, Sun-2019-Generalized}.
In addition, we define
\begin{equation}\label{def:Gamma_f}
\Gamma_f(x) = \sup\left\{t \geq 0: t-\log(1+t) \leq \kappa^2\left(f(x)-\inf_{z \in \br^n} f(z)\right)\right\}.
\end{equation}
\begin{proposition}\label{prop:descent}
Suppose that $f$ is $F$-based $\kappa$-self-concordant and $F$ is convex. Then, for any $x \in \dom(f)$ and any descent direction
$d \in \br^n$ satisfying $\nabla f(x)^\top d 
\leq 0$,
defining
\begin{equation*}
\rho = -\nabla f(x)^\top d,\quad \delta = \nabla^2 f(x)[d,d],\quad  \Delta = \nabla^2 F(x)[d,d],\quad \eta = \tfrac{\rho}{\sqrt{\Delta+\delta}},
\end{equation*}
we have, for all
$t \in [0, \frac{1}{\kappa \sqrt{\delta+\Delta}})$, that $x+td \in \dom(f)$ and
\begin{equation}\label{eq:descent-convex}
f(x+td) \leq f(x) - \rho t + \kappa^{-2}\omega_\star(\kappa t\sqrt{\delta + \Delta}).
\end{equation}
Letting $\bar{t} = \frac{\rho}{\delta+\Delta+\kappa\rho\sqrt{\delta+\Delta}}$, we have $f(x+td) \leq f(x)$ for all $t \in [0, \bar{t}]$ and
\begin{equation}\label{eq:descent-convex-optimal}
f(x+\bar{t}d) \leq f(x) - \rho\bar{t} + \kappa^{-2}\omega_\star(\kappa\bar{t}\sqrt{\delta+\Delta}) = f(x)-\kappa^{-2}\omega(\kappa\eta) \leq f(x) - \tfrac{\eta^2}{2(1+\Gamma_f(x))}.
\end{equation}
Moreover, if $F$ is not only convex , but is also $\kappa_F$-self-concordant, we have
\begin{equation}\label{eq:descent-SC}
f(x+td) \leq f(x) - \rho t + \kappa^{-2}\omega_\star(\kappa t\sqrt{\delta + \Delta}) - \kappa_F^{-2}\omega(\kappa_F t\sqrt{\Delta}).
\end{equation}
\end{proposition}
\begin{proof}{Proof}
It follows from \cite[Theorems~5.1.5 and 5.1.9]{Nesterov-2018-Lectures} that $x+td \in \dom(f)$ and
\begin{equation*}
f(x+td) - f(x) + \rho t \leq F(x) + t\nabla F(x)^\top d - F(x+td) + \kappa^{-2}\omega_\star(\kappa t\sqrt{\delta + \Delta}).
\end{equation*}
Since $F$ is convex, we have $F(x) + t\nabla F(x)^\top d - F(x+td) \leq 0$. Plugging this into the above inequality yields 
inequality \eqref{eq:descent-convex}. We define
\begin{equation*}
g(t) = f(x) - \rho t - \tfrac{1}{\kappa^2}\left(\kappa t\sqrt{\delta + \Delta} + \log(1-\kappa t\sqrt{\delta + \Delta})\right).
\end{equation*}
Taking the derivative of $g(t)$ yields
\begin{equation*}
g'(t) = - \rho + \left(\tfrac{t(\delta + \Delta)}{1-\kappa t\sqrt{\delta + \Delta}}\right).
\end{equation*}
It is easy to verify from the definition of $\bar{t}$ that
 $g'(t) \leq 0$ for all $t \in [0,\bar{t}]$. Since $g(0)=f(x)$, we have $f(x+td) \leq g(t) \leq f(x)$ for all $t \in [0,\bar{t}]$. Plugging $\bar{t}$ into \eqref{eq:descent-convex} yields
\eqref{eq:descent-convex-optimal}. In addition,
\eqref{eq:descent-convex-optimal} implies
\begin{equation*}
\kappa \eta - \log(1 + \kappa \eta) = \omega(\kappa\eta) \leq \kappa^2(f(x) - f(x+\bar{t}d)) \leq \kappa^2\left(f(x)-\inf_{z \in \br^n} f(z)\right).
\end{equation*}
Thus, $\kappa\eta \leq \Gamma_f(x)$. Since $\omega(z) \geq \frac{z^2}{2(1+\Gamma_f(x))}$ for $0 \leq z \leq \Gamma_f(x)$, we have $\kappa^{-2}\omega(\kappa\eta) \geq \frac{\eta^2}{2(1+\Gamma_f(x))}$. It follows from \cite[Theorem~5.1.8]{Nesterov-2018-Lectures} that
\begin{equation*}
F(x) + t\nabla F(x)^\top d - F(x+td) \leq - \kappa_F^{-2}\omega(\kappa_F t\sqrt{\Delta}).
\end{equation*}
which implies
inequality
\eqref{eq:descent-SC}. \Halmos
\end{proof}

\section{Algorithms}
\label{sec:algorithms}
We develop two algorithms, Algorithm~\ref{algorithm:RNM} (RNM) and Algorithm~\ref{algorithm:ARM} (ARM), for
minimizing an unconstrained function $f$ (i.e., solving problem~\eqref{prob:unconstrained}),
where $f$ is $F$-based $\kappa$-self-concordant and bounded below, and present proofs of their convergence. Compared to
other regularized methods (e.g., see ~\cite{Polyak-2009-Regularized, Nesterov-2006-Cubic, Cartis-2011-Adaptive,Cartis-2011-Adaptive2, Gratton-2025-Yet}),
our
algorithms assume neither $\dom(f) = \br^n$ nor
a Lipschitz Hessian condition.
Throughout
the rest of
this paper, we make Assumption~\ref{Assumption:main}
below
and impose additional assumptions (e.g., $F$ is self-concordant or quadratic) when necessary.

\begin{assumption}\label{Assumption:main}
The function $f: \br^n \to \br \cup \{+\infty\}$ is bounded below and is $F$-based $\kappa$-self-concordant with respect to a convex function $F: \br^n \to \br \cup \{+\infty\}$ for some $\kappa > 0$.
\end{assumption}

We also make use of the following definitions in the rest of this section:
\begin{equation*}
\rho_j = -\nabla f(x_j)^\top d_j,\quad \delta_j = \nabla^2 f(x_j)[d_j,d_j],\quad  \Delta_j = \nabla^2 F(x_j)[d_j,d_j],\quad \eta_j = \tfrac{\rho_j}{\sqrt{\Delta_j+\delta_j}}.
\end{equation*}
\subsection{Regularized Newton's Method (RNM)} \label{subsec:RNM}

At each iteration,
Algorithm~\ref{algorithm:RNM} (RNM) computes
a direction $d_j \in \br^n$ that minimizes the right-hand side of
inequality~\eqref{eq:descent-convex-optimal}. This is equivalent to maximizing $\eta_j = \tfrac{-\nabla f(x_j)^\top d_j}{\sqrt{d_j^\top(\nabla^2(f+F)(x_j))d_j}}$ w.r.t. $d_j$ and yields Algorithm~\ref{algorithm:RNM}.

\begin{algorithm}
\caption{Regularized Newton's Method}\label{algorithm:RNM}
\vskip6pt
\begin{algorithmic}[1]
\State \textbf{Initialization:} $x_0 \in \dom(f)$.
\For{$j=0,1,2,\ldots$}
\State $d_j \gets -(\nabla^2(f+F)(x_j))^{-1} \nabla f(x_j)$.
\State $\lambda_j \gets \sqrt{-\nabla f(x_j) ^\top d_j}$.
\State $t_j \gets \tfrac{1}{1+\kappa \lambda_j}$.
\State $x_{j+1} \gets x_j + t_j d_j$.
\EndFor
\end{algorithmic}
\end{algorithm}
Let $A^+$ denote the pseudoinverse of $A\in\br^{n\times n}$ and $\sqrt{x} = \infty$ if $x<0$. These conventions will be convenient in Algorithm~\ref{algorithm:ARM}. To analyze Algorithm~\ref{algorithm:RNM}, we define the square root of regularized Newton decrement as
\begin{equation}
\label{eq:def-nu}
\nu_{f,F}(x_j)=\sqrt{\delta_j+\Delta_j}, \; \; \mathrm{where}\ d_j=-(\nabla^2(f+F)(x_j))^{+}\nabla f(x_j),
\end{equation}
Under Assumption~\ref{Assumption:main}, we derive
convergence guarantees in terms of $\nu_{f,F}(x)$ rather than
the
gradient norm, which is insufficient to measure the stationarity due to the lack of Lipschitz gradients. Our results can be interpreted as a generalization of \cite[Section 5.2]{Nesterov-2018-Lectures} and $\nu_{f,F}(x) = \lambda_f(x)$ when $F=0$ (see \cite[Page 354]{Nesterov-2018-Lectures}). For functions that are weakly self-concordant, this measure relates to gradient norms of Moreau envelopes (see Appendix \ref{subsec:convergence-to-a-proximal-stationary-point}). We summarize our results in the following two theorems. 
\begin{theorem}\label{thm:RNM}
Suppose that Assumption~\ref{Assumption:main} holds. Then, the sequence of iterates $\{x_j\}_{j \geq 0}$ generated by Algorithm~\ref{algorithm:RNM} satisfy
\begin{equation*}
\min_{0 \leq j \leq k} \nu_{f,F}(x_j) \leq \sqrt{\tfrac{2(1+\Gamma_f(x_0))(f(x_0)-\inf_{x \in \br^n} f(x))}{k+1}}.
\end{equation*}
where $\Gamma_f(x)$ is defined in Eq.~\eqref{def:Gamma_f}.
\end{theorem}
\begin{proof}{Proof}
    Proposition~\ref{prop:descent} guarantees that $x_j\in\dom(f)$, $f(x_j) \leq f(x_0)$ for all $j\geq 0$ and $f(x_j)-f(x_{j+1}) \geq \frac{(\nu_{f,F}(x_j))^2}{2(1+\Gamma_f(x_0))}$. Summing this inequality over $j = 0, 1, 2, \ldots$ yields
\begin{equation*}
f(x_0) - \inf_{x \in \br^n} f(x) \geq \sum_{j=0}^k f(x_j)-f(x_{j+1}) \geq \sum_{j=0}^k \tfrac{\nu_{f,F}(x_j)^2}{2(1+\Gamma_f(x_0))}\geq \tfrac{(k+1)(\min_{0\leq j\leq k} \nu_{f,F}(x_j))^2}{2(1+\Gamma_f(x_0))}.
\end{equation*}
This completes the proof. \Halmos
\end{proof}
If the Polyak-\L{}ojasiewicz (PL) condition \cite{Lojasiewicz-1963-Propriete, Polyak-1964-Gradient} holds in
a neighborhood of
a limit point of $\{x_j\}_{j \geq 0}$, we further have
\begin{theorem}
\label{thm:local-convergence-RNM}
Suppose that Assumption~\ref{Assumption:main} holds, $f+F$ is $\alpha$-strongly convex and
$f$ satisfies  the Polyak-\L{}ojasiewicz (PL) condition with constant $\mu > 0$ in a neighborhood $U$ in $\dom(f)$ of an
optimal solution  $x^\star$,
(i.e., $\frac{1}{2}
\|\nabla f(x)\|^2 \geq \mu(f(x) - f(x^\star))$ for all $x
\in U$).
If $x^\star$ is a limit point of $\{x_j\}_{j \geq 0}$,  generated by Algorithm~\ref{algorithm:RNM},
we have $x_j \rightarrow x^\star$ and $f(x_j) - f(x^\star)$ converges $Q$-linearly to $0$.
\end{theorem}
\begin{proof}{Proof}
In
\cite[Proposition 3.3 and Theorem 3.4]{Absil-2005-Convergence}, 
it is proved that if $f$ satisfies the PL condition in a neighborhood $U$  of $x_\star$ then $x_j \to x_\star$ if there exist $c_1,c_2>0$ such that $f(x_j) - f(x_{j+1}) \geq c_1 \|\nabla f(x_j)\|^2$ and $\|\nabla f(x_j)\| \geq c_2\|x_j - x_{j+1}\|$ for all $x_j,x_{j+1} \in U$. 
Now by Proposition~\ref{prop:descent} and
the facts that $f(x_j) \leq f(x_0)$ and $\|\nabla^2(f+F)(x_j)\|_{\mathrm{op}}\leq L$,
for some $L>0$, for all $x_j \in U$, we have
\begin{equation} \label{eq:PL-descent-lemma}
f(x_j) - f(x_{j+1}) \geq \tfrac{1}{2(1+\Gamma_f(x_0))} \left(\nu_{f,F}(x_j)\right)^2 \geq \tfrac{1}{2(1+\Gamma_f(x_0)) L} \|\nabla f(x_j)\|^2, \textnormal{ for any } x_j \in U.
\end{equation}
We also have $\|\nabla f(x_j)\| \geq \alpha\|x_{j+1}-x_j\|$.
Hence, from the results in \cite{Absil-2005-Convergence} stated above, it follows that, $x_j \to x_\star$.

It remains to prove the linear convergence of function values.
Since $x_j \to x^\star$, there exists $K \in \mathbb{N}$ such that $x_j \in U$ for all $j \geq K$. Applying Eq.~\eqref{eq:PL-descent-lemma} and the
PL inequality yields $f(x_j) - f(x_{j+1}) \geq \frac{\mu}{(1+\Gamma_f(x_0))L}(f(x_j) - f(x^\star))$ for all $j \geq K$. This implies that
\begin{equation*}
f(x_{j+1})-f(x^\star) \leq \left(1 - \tfrac{\mu}{(1+\Gamma_f(x_0))L}\right)(f(x_j)-f(x^\star)),
\end{equation*}
which implies that $f(x_j)-f(x^\star)$ converges $Q$-linearly to $0$. \Halmos
\end{proof}

\subsection{Adaptive Regularization Method (ARM)}
\begin{algorithm}[!t]
\caption{Adaptive Regularization Method (ARM)}\label{algorithm:ARM}
\vskip6pt
\begin{algorithmic}[1]
\State \textbf{Input:} $0 < \sigma_{\min} \leq \sigma_0$, $0 < \eta_1 \leq \eta_2 < 1$, $0 < \gamma_1 < 1 < \gamma_2 < \gamma_3$. 
\textbf{Opt} = 1 \text{or} 2.
\State \textbf{Initialization:} $x_0 \in \dom(f)$.
\For{$j=0,1,2,\ldots$}
\State
\textbf{Opt}=1:
Given $H_j$, compute $d_j$ and $m_j$ by Eq.~\eqref{eq:d-and-m-ARM-1st}, $\rho_j\gets -\nabla f (x_j)^\top d_j$, $\delta_j \gets d_j^\top \nabla^2f(x_j)d_j$, and $\Delta_j \gets d_j^\top \nabla^2F(x_j)d_j$. \textbf{return} $x_j$ \textbf{if} ${\rho_j} \leq \epsilon \sqrt{\delta_j+\sigma_j\Delta_j}$ \textbf{and} $\rho_j \geq 0$ \textbf{and} $\delta_j+\sigma_j\Delta_j\geq 0$.
\State
\textbf{Opt}=2: Compute $v_j$ s.t. $\|v_j\|=1$ and $\nabla^2 f(x_j)v_j = \lambda_{\min}(\nabla^2f(x_j))v_j$, $d_j$ by
\eqref{eq:d-negative-curvature}, $m_j$ by
\eqref{eq:m-negative-curvature}, $\rho_j\gets -\nabla f(x_j)^\top d_j$, $\delta_j \gets d_j^\top \nabla^2f(x_j)d_j$, and $\Delta_j \gets d_j^\top \nabla^2F(x_j)d_j$. \textbf{return} $x_j$ \textbf{if} $\nu_{f,\sigma_jF}(x_j)\leq \epsilon_g$ and $\lambda_{\min}(\nabla^2f(x_j)) \geq -\sigma_j\sqrt{\epsilon_H} v_j^\top\nabla^2F(x_j)v_j$.
\State $t_j\gets\argmin_{t \geq 0} m_j(t)$, $t_j\gets0$ \textbf{if} $t_j=\infty$ \textbf{or} any tie occurs.
\State $r_j \gets \frac{f(x_j)-f(x_j+t_jd_j)}{f(x_j)-m_j(t_j)}$, where $f(x_j+t_jd_j)\gets +\infty$ if $x_j+t_jd_j\notin \dom(f)$ and $\frac00=0$.
\begin{equation*}
x_{j+1} \gets \begin{cases}
x_j+t_jd_j, & \textnormal{if } r_j > \eta_1,\\
x_j, & \textnormal{otherwise}.
\end{cases} \quad
\sigma_{j+1} \in \begin{cases}
[\max(\sigma_{\min},\gamma_1\sigma_j),\sigma_j], & \textnormal{if } r_j \geq \eta_2,\\
[\sigma_j,\gamma_2\sigma_j], & \textnormal{if } r_j \in (\eta_1,\eta_2),\\
[\gamma_2\sigma_j,\gamma_3\sigma_j], & \mathrm{if}\ r_j\leq\eta_1.
\end{cases}
\end{equation*}
\EndFor
\end{algorithmic}
\end{algorithm}
Our second algorithm Algorithm~\ref{algorithm:ARM} (ARM)
is similar in spirit to the methods in \cite{Cartis-2011-Adaptive,Cartis-2011-Adaptive2,Cartis-2022-Evaluation}.
At each iteration, Algorithm~\ref{algorithm:ARM} computes a direction $d_j \in \br^n$ and a step-size $t_j$ by minimizing the model $m_j(t)$. It then computes the ratio $r_j$ to determine if a sufficient decrease is made. If $r_j > \eta_1$, it accepts the update; otherwise, the iterate is not updated. Finally, it adjusts the parameter $\sigma_j$ based on $r_j$.


Algorithm~\ref{algorithm:ARM} has two options (\textbf{Opt} $= 1$ or $2$)
for selecting the search direction $d_j$ at \textbf{every} iteration.
The first option is preconditioned gradient descent, given by
\begin{equation}
\label{eq:d-and-m-ARM-1st}
d_j=-H_j\nabla f(x_j),\quad  m_j:\br_{+}\to \br, m_j(t)=f(x_j)- \rho_j t + \kappa^{-2}\,\omega_\star(\kappa t\sqrt{\delta_j + \sigma_j \Delta_j}),
\end{equation}
The convergence of option 1 of ARM depends on the choice of $H_j$. Our convergence analysis (Theorem~\ref{thm:ARM-1st} and Lemma~\ref{lem:preconditioned-GD}) covers gradient descent ($H_j = I$), quasi-Newton methods, and Newton-type methods. Under the assumptions of Lemmas~\ref{lem:preconditioned-GD}~or~\ref{lem:ARNM}, we are able to show that option 1 of ARM returns
an approximate first-order stationary point. The second option of ARM, which requires an oracle that detects negative curvature, returns
an
approximate second-order stationary point (see Theorem~\ref{thm:ARM-2nd}). In what follows, we provide
additional details and convergence results.

\begin{theorem} \label{thm:ARM-1st}
Suppose that Assumption~\ref{Assumption:main} holds. If there exists $\sigma_{\max}>0$ such that $\sigma_j \leq \sigma_{\max}$ for all $j\geq 0$, then
Algorithm~\ref{algorithm:ARM} (ARM) with
\textbf{Opt} $= 1$
terminates within at most $O(\epsilon^{-2})$ iterations.
\end{theorem}
\begin{proof}{Proof}
Fix an integer $k\ge 1$ and suppose Algorithm~\ref{algorithm:ARM} has not terminated before iteration $k$.
Let $\mathcal{S}_k$ and $\mathcal{U}_k$ denote the sets of \emph{successful} and \emph{unsuccessful} iterations up to $k$ in Algorithm~\ref{algorithm:ARM},
\[
\mathcal{S}_k:=\{0\le j\le k-1:\ r_j> \eta_1\},\qquad
\mathcal{U}_k:=\{0\le j\le k-1:\ r_j\leq\eta_1\}.
\]
For each $j\in \mathcal{S}_k$, we have $x_{j+1}=x_j+t_jd_j$ and therefore
\[
f(x_j)-f(x_{j+1})
= r_j\bigl(f(x_j)-m_j(t_j)\bigr)
\ge \eta_1\bigl(f(x_j)-m_j(t_j)\bigr).
\]
Hence, if there exists a constant $\widetilde\epsilon>0$ such that $f(x_j)-m_j(t_j)\ge \widetilde\epsilon$ for all $j\in\mathcal{S}_k$, then
\[
f(x_0)-\inf_{z\in\br^n}f(z)
\ge \sum_{j\in\mathcal{S}_k}\bigl(f(x_j)-f(x_{j+1})\bigr)
\ge \eta_1\,\lvert \mathcal{S}_k\rvert\,\widetilde\epsilon,
\]
which implies $\lvert \mathcal{S}_k\rvert = O(\widetilde\epsilon^{-1})$.

We now lower bound the model decrease on successful iterations. For
option 1 in Algorithm~\ref{algorithm:ARM}, whenever $j\in\mathcal{S}_k$, a direct calculation from the explicit form of $m_j$
yields the bound
\[
f(x_j)-m_j(t_j)\ \ge\ \kappa^{-2}\omega \left(\tfrac{-\kappa\nabla f(x_j)^\top d_j}{\sqrt{\nabla^2(f+\sigma_jF)(x_j)[d_j,d_j]}}\right).
\]
Since the algorithm has not terminated, the
option 1 termination condition is violated at $x_j$, i.e.,
\[
\tfrac{\rho_j}{\sqrt{\delta_j+\sigma_j\Delta_j}} =
\tfrac{-\nabla f(x_j)^\top d_j}{\sqrt{\nabla^2(f+\sigma_jF)(x_j)[d_j,d_j]}}>\epsilon.
\]
Consequently, for all $j\in\mathcal{S}_k$,
\[
f(x_j)-m_j(t_j)>\kappa^{-2}\omega(\kappa \epsilon),
\]
and hence $\lvert \mathcal{S}_k\rvert = O(\epsilon^{-2})$.

Finally, we bound the number of unsuccessful iterations using the update rule for $\sigma_j$.
For any $j\in \mathcal{S}_k$ we have $\sigma_{j+1}\ge \max(\sigma_{\min},\gamma_1\sigma_j)\ge \gamma_1\sigma_j$, while for any $j\in\mathcal{U}_k$ we have $\sigma_{j+1}\ge \gamma_2\sigma_j$.
Thus,
\[
\sigma_0\,\gamma_1^{\lvert \mathcal{S}_k\rvert}\gamma_2^{\lvert \mathcal{U}_k\rvert}\ \le\ \sigma_k\ \le\ \sigma_{\max}.
\]
Taking logarithms and rearranging yields
\[
\lvert \mathcal{U}_k\rvert \le \tfrac{\log(\sigma_{\max}/\sigma_0)-\lvert \mathcal{S}_k\rvert\log(\gamma_1)}{\log(\gamma_2)},
\]
and therefore $k=\lvert\mathcal{S}_k\rvert+\lvert\mathcal{U}_k\rvert = O(\lvert\mathcal{S}_k\rvert)=O(\epsilon^{-2})$. \Halmos
\end{proof}

In many practical settings the preconditioner satisfies $H_j \succ 0$, e.g., gradient descent with $H_j = I$, quasi-Newton methods \cite{Li-2001-Global}, and inverse Gauss--Newton approximations. Under this additional structure, a spectral bound on the preconditioned Hessian ensures that $\sigma_j$ remains uniformly bounded,
as shown in Lemma~\ref{lem:preconditioned-GD} below,
so the $O(\epsilon^{-2})$ complexity of Theorem~\ref{thm:ARM-1st} applies directly.
\begin{lemma} \label{lem:preconditioned-GD}
Suppose that Assumption~\ref{Assumption:main} holds and $H_j\succ 0$. Then, we have $\sigma_j \leq \max(\sigma_0, \gamma_3)$ for ARM with \textbf{Opt} $=1$.
\end{lemma}
\begin{proof}{Proof}
With $d_j=-H_j\nabla f(x_j)$ and $H_j\succ 0$, we have $\nabla f(x_j)^\top d_j=-\nabla f(x_j)^\top H_j\nabla f(x_j)\le 0$ for all $j$.
Suppose for contradiction that there exists an index $j$ such that $\sigma_{j+1}>\max(\sigma_0,\gamma_3)\ge \sigma_j$.
Then $\sigma_j\ge \sigma_{j+1}/\gamma_3\ge 1$.
Assumption~\ref{Assumption:main} implies $\delta_j+\Delta_j\ge 0$, and hence $\delta_j+\sigma_j\Delta_j\ge 0$.
By Proposition~\ref{prop:descent}, we have $x_j+t_jd_j\in\dom(f)$ and
\[
f(x_j)-f(x_j+t_jd_j)\ \ge\ \rho_j t_j-\kappa^{-2}\omega_\star\!\bigl(\kappa t_j\sqrt{\delta_j+\Delta_j}\bigr).
\]
On the other hand, by the definition of the
option 1 model $m_j$,
\[
f(x_j)-m_j(t_j)=\rho_j t_j-\kappa^{-2}\omega_\star\!\bigl(\kappa t_j\sqrt{\delta_j+\sigma_j\Delta_j}\bigr).
\]
Since $\sigma_j\ge 1$ and $\omega_\star$ is increasing on $\br_+$, we obtain $f(x_j)-f(x_j+t_jd_j)\ge f(x_j)-m_j(t_j)$, so $r_j\ge 1>\eta_2$.
Algorithm~\ref{algorithm:ARM} would therefore enforce $\sigma_{j+1}\le \sigma_j$, contradicting $\sigma_{j+1}>\sigma_j$.
Hence, $\sigma_j\le \max(\sigma_0,\gamma_3)$ for all $j$. \Halmos
\end{proof}

We can further specialize ARM to obtain an adaptive method by leveraging the self-concordant structure of $f + \sigma_j F$ for $\sigma_j\geq 1$. In particular, we define
\begin{equation} \label{eq:d-ARNM}
d_j = -(\nabla^2(f+\sigma_j F)(x_j))^{+} \nabla f(x_j),
\end{equation}
which yields
\begin{equation} \label{eq:m-ARNM}
m_j(t) = f(x_j) - t(\nu_{f,\sigma_j F}(x_j))^2 + \kappa^{-2} \omega_\star(\kappa t \nu_{f,\sigma_j F}(x_j)).
\end{equation}
The termination condition becomes $\nu_{f,\sigma_j F}(x_j) \leq \epsilon$. If $\delta_j + \sigma_j \Delta_j > 0$ and $\rho_j \geq 0$, the optimal step size is
\begin{equation} \label{eq:t-ARA-2nd-nu}
t_j = \tfrac{1}{1 + \kappa \nu_{f,\sigma_j F}(x_j)}.
\end{equation}
\begin{lemma}\label{lem:ARNM}
Suppose that Assumption~\ref{Assumption:main} holds. Let $d_j$ be defined by Eq.~\eqref{eq:d-ARNM}. Then, we have $\sigma_j \leq \max(\sigma_0, 1, \gamma_3)$ for ARM with \textbf{Opt} $=1$.
\end{lemma}
\begin{proof}{Proof}
Suppose for contradiction that there exists $j$ such that
$\sigma_{j+1}>\max(\sigma_0,1,\gamma_3)\ge \sigma_j$.
Since $\sigma_{j+1}>\gamma_3$ and Algorithm~\ref{algorithm:ARM} always enforces
$\sigma_{j+1}\le \gamma_3\sigma_j$, we have \(\sigma_j\ge\frac{\sigma_{j+1}}{\gamma_3}>1\).
Let $\rho_j:=-\nabla f(x_j)^\top d_j\ge 0$, $\delta_j:=\nabla^2 f(x_j)[d_j,d_j]$, and
$\Delta_j:=\nabla^2F(x_j)[d_j,d_j]\ge 0$.
Assumption~\ref{Assumption:main} implies $\delta_j+\Delta_j\ge 0$ and hence
$\delta_j+\sigma_j\Delta_j\ge 0$.
By Proposition~\ref{prop:descent}, we have $x_j+t_jd_j\in\dom(f)$ and
\[
f(x_j)-f(x_j+t_jd_j)
\ \ge\
\rho_j t_j-\kappa^{-2}\omega_\star\!\bigl(\kappa t_j\sqrt{\delta_j+\Delta_j}\bigr).
\]
On the other hand, since $d_j=-(\nabla^2(f+\sigma_jF)(x_j))^{+}\nabla f(x_j)$,
we have
from (\ref{eq:def-nu}) that $\rho_j= \nabla f(x_j)^\top\bigl(\nabla^2(f+\sigma_jF)(x_j)\bigr)^{+}\nabla f(x_j)
= \bigl(\nu_{f,\sigma_jF}(x_j)\bigr)^2 = \delta_j+\sigma_j\Delta_j$, and hence
from \eqref{eq:m-ARNM} that
\[
f(x_j)-m_j(t_j)
= \rho_j t_j-\kappa^{-2}\omega_\star\!\bigl(\kappa t_j\sqrt{\delta_j+\sigma_j\Delta_j}\bigr).
\]
Because $\sigma_j\ge 1$ and $\omega_\star$ is increasing on its domain, we have
\(\omega_\star\!\bigl(\kappa t_j\sqrt{\delta_j+\sigma_j\Delta_j}\bigr)
\ge\omega_\star\!\bigl(\kappa t_j\sqrt{\delta_j+\Delta_j}\bigr)\);
therefore $f(x_j)-f(x_j+t_jd_j)\ge f(x_j)-m_j(t_j)$, i.e.,
$r_j\ge 1>\eta_2$.
By the update rule in Algorithm~\ref{algorithm:ARM}, $r_j\ge \eta_2$ forces
$\sigma_{j+1}\le \sigma_j$, contradicting $\sigma_{j+1}>\sigma_j$. \Halmos
\end{proof}

Suppose that $F$ is $\kappa_F$-self-concordant. Then, we exploit negative curvature to escape from saddle points and further decrease the objective. For each iterate $x_j$, we have access to the smallest eigenvalue of Hessian $\lambda_{\min}(\nabla^2 f(x_j))$ and a corresponding unit eigenvector $v_j$. If the Hessian exhibits strong negative curvature, Algorithm~\ref{algorithm:ARM} switches from a regularized Newton step to a curvature-exploiting step as follows,
\begin{equation} \label{eq:d-negative-curvature}
d_j = \begin{cases}
-(\nabla^2(f+\sigma_jF)(x_j))^{+}\nabla f(x_j),& \textnormal{if }\lambda_{\min}(\nabla^2 f(x_j)) \geq -\lambda_{\mathrm{nc}},\\
v_j, & \textnormal{if } \nabla f(x_j)^\top v_j\leq 0,\ \lambda_{\min}(\nabla^2 f(x_j)) < -\lambda_{\mathrm{nc}},\\
-v_j, & \textnormal{otherwise,}
\end{cases}
\end{equation}
where $\lambda_{\mathrm{nc}}=\sigma_j\sqrt{\epsilon_H}\nabla^2F(x_j)[v_j,v_j]$. This rule ensures that, if the significant negative curvature is detected, a descent direction aligned with it will be used to escape saddle regions.
We define the local model $m_j(t)$ as follows,
\begin{equation} \label{eq:m-negative-curvature}
m_j(t) = \begin{cases}
f(x_j) - t(\nu_{f,\sigma_j F}(x_j))^2 + \kappa^{-2} \omega_\star(\kappa t \nu_{f,\sigma_j F}(x_j)),&\text{if }\lambda_{\min}(\nabla^2 f(x_j))\ge -\lambda_{\mathrm{nc}},\\
f(x_j)-\kappa_F^{-2}\omega(\kappa_F t\sqrt{\sigma_j\Delta_j})+\kappa^{-2}\omega_\star(\kappa t\sqrt{\delta_j+\sigma_j\Delta_j}),&\text{otherwise.}
\end{cases}
\end{equation}
If $\lambda_{\min}(\nabla^2 f(x_j))\geq -\lambda_{\mathrm{nc}}$, we have that $m_j(t)$ reduces to the one used for adaptive regularized Newton's methods (see Eq.~\eqref{eq:t-ARA-2nd-nu}). Thus, we can use Eq.~\eqref{eq:t-ARA-2nd-nu}. If $\lambda_{\min}(\nabla^2 f(x_j))<-\lambda_{\mathrm{nc}}$ and $\delta_j+\sigma_j\Delta_j\ge 0$, Algorithm~\ref{algorithm:ARM} performs a negative curvature step where $\delta_j=\lambda_{\min}(\nabla ^2f(x_j))<0$. Minimizing $m_j(t)$ with respect to $t$ yields
\begin{equation} \label{eq:t-negative-curvature-descent}
t_j =
\begin{cases}
-\tfrac{\delta_j}{\sqrt{\sigma_j\Delta_j}\sqrt{\delta_j+\sigma_j\Delta_j}(\kappa_F\sqrt{\delta_j+\sigma_j\Delta_j}+\kappa\sqrt{\sigma_j\Delta_j})}, & \mathrm{if}\ \lambda_{\min}(\nabla^2 f(x_j))<-\lambda_{\mathrm{nc}}, \\
\tfrac{1}{1 + \kappa \nu_{f,\sigma_j F}(x_j)}, & \mathrm{if}\ \lambda_{\min}(\nabla^2 f(x_j)) \geq -\lambda_{\mathrm{nc}}.
\end{cases}
\end{equation}
\begin{theorem} \label{thm:ARM-2nd}
Suppose that Assumption~\ref{Assumption:main} holds with a $\kappa_F$-self-concordant and strictly convex $F$ and $d_j$ and $m_j$ are defined by Eqs.~\eqref{eq:d-negative-curvature}~and~\eqref{eq:m-negative-curvature} for all $j$. 
Then, at all iterations of ARM with \textbf{Opt} $=2$, we have $\sigma_j \leq \max(\sigma_0,\gamma_3)$ for all $j$, and the option 2 variant of Algorithm~\ref{algorithm:ARM} terminates within at most $O(\epsilon_g^{-2}+\epsilon_H^{-1.5})$ iterations.
\end{theorem}

\begin{proof}{Proof}
We prove that $\sigma_j$ is uniformly upper bounded by contradiction.
Assume there exists an index $j$ such that $\sigma_{j+1}>\max(\sigma_0,\gamma_3)\ge \sigma_j$.
By the update rule for $\sigma$ in Algorithm~\ref{algorithm:ARM}, we always have $\sigma_{j+1}\le \gamma_3\sigma_j$; hence
$\sigma_j\ge \sigma_{j+1}/\gamma_3>1$.
Assumption~\ref{Assumption:main} implies $\delta_j+\Delta_j\ge 0$ and hence $\delta_j+\sigma_j\Delta_j\ge 0$.

We consider two cases. \emph{Case 1: $\lambda_{\min}(\nabla^2 f(x_j))\ge -\lambda_{\mathrm{nc}}$.}
Then $d_j$ is the regularized Newton direction and $m_j$ reduces to the Newton-type model.
By Proposition~\ref{prop:descent}, we have $x_j+t_jd_j\in\dom(f)$ and
\[
f(x_j)-f(x_j+t_jd_j)\ \ge\ \rho_j t_j-\kappa^{-2}\omega_\star\!\bigl(\kappa t_j\sqrt{\delta_j+\Delta_j}\bigr).
\]
Comparing with the definition of $m_j$ in this case and using $\sigma_j\ge 1$ together with the monotonicity of $\omega_\star$ yields
$f(x_j)-f(x_j+t_jd_j)\ge f(x_j)-m_j(t_j)$, i.e., $r_j\ge 1>\eta_2$.
Algorithm~\ref{algorithm:ARM} would therefore enforce $\sigma_{j+1}\le \sigma_j$, contradicting $\sigma_{j+1}>\sigma_j$. \emph{Case 2: $\lambda_{\min}(\nabla^2 f(x_j))<-\lambda_{\mathrm{nc}}$.}
Then $d_j=\pm v_j$ and $\delta_j=\lambda_{\min}(\nabla^2 f(x_j))<0$.
In this case,
\[
f(x_j)-m_j(t_j)
= -\kappa^{-2}\omega_\star\!\bigl(\kappa t_j\sqrt{\delta_j+\sigma_j\Delta_j}\bigr)
+\kappa_F^{-2}\omega\!\bigl(\kappa_F t_j\sqrt{\sigma_j\Delta_j}\bigr).
\]
By Proposition~\ref{prop:descent}, we also have
\[
f(x_j)-f(x_j+t_jd_j)
\ge -\kappa^{-2}\omega_\star\!\bigl(\kappa t_j\sqrt{\delta_j+\Delta_j}\bigr)
+\kappa_F^{-2}\omega\!\bigl(\kappa_F t_j\sqrt{\Delta_j}\bigr).
\]
We define \(g(\sigma):=-\kappa^{-2}\omega_\star\!\bigl(\kappa t_j\sqrt{\delta_j+\sigma\Delta_j}\bigr)
+\kappa_F^{-2}\omega\!\bigl(\kappa_F t_j\sqrt{\sigma\Delta_j}\bigr)\). A direct differentiation shows
\[
g'(\sigma)=\tfrac{t_j^2\Delta_j}{2}\Bigl(-\tfrac{1}{1-\kappa t_j\sqrt{\delta_j+\sigma\Delta_j}}+\tfrac{1}{1+\kappa_F t_j\sqrt{\sigma\Delta_j}}\Bigr)\le 0,
\]
and
hence, $g(\sigma_j)\le g(1)$.
Therefore $f(x_j)-m_j(t_j)=g(\sigma_j)\le g(1)\le f(x_j)-f(x_j+t_jd_j)$, which implies $r_j\ge 1>\eta_2$ as in Case 1. Therefore,
We conclude that $\sigma_j\le \max(\sigma_0,\gamma_3)$ for all $j\ge 0$.

The argument below follows the usual successful/unsuccessful iteration counting used in adaptive regularization \cite{Cartis-2011-Adaptive, Cartis-2011-Adaptive2} and trust-region analyses \cite{Wright-2006-Numerical}; we include it in full for completeness.
Fix an integer $k\ge 1$ and suppose Algorithm~\ref{algorithm:ARM} has not terminated before iteration $k$.
Define the sets of accepted (successful) and rejected (unsuccessful) iterations up to $k$ by
\[
\mathcal{S}_k:=\{0\le j\le k-1:\ r_j> \eta_1\},
\qquad {\rm and} \qquad
\mathcal{U}_k:=\{0\le j\le k-1:\ r_j\leq \eta_1\}.
\]

For any $j\in\mathcal{S}_k$, the step is accepted, so $x_{j+1}=x_j+t_jd_j$ and
\[
f(x_j)-f(x_{j+1})
= f(x_j)-f(x_j+t_jd_j)
= r_j\bigl(f(x_j)-m_j(t_j)\bigr)
\ge \eta_1\bigl(f(x_j)-m_j(t_j)\bigr).
\]
For any $j\in\mathcal{U}_k$, the step is rejected, so $x_{j+1}=x_j$ and $f(x_j)-f(x_{j+1})=0$.
Let $f_{\inf}:=\inf_{x\in\br^n} f(x)$ (finite under Assumption~\ref{Assumption:main}).
Summing over $j=0,\dots,k-1$ yields
\begin{equation}
\label{eq:telescoping-ARM-2nd}
f(x_0)-f_{\inf}
\ge \textstyle\sum_{j=0}^{k-1}f(x_j)-f(x_{j+1})
= \textstyle\sum_{j\in\mathcal{S}_k}f(x_j)-f(x_{j+1})
\ge \eta_1\textstyle\sum_{j\in\mathcal{S}_k}f(x_j)-m_j(t_j).
\end{equation}

We now lower bound $f(x_j)-m_j(t_j)$ for $j\in\mathcal{S}_k$.
Since the algorithm has not terminated at $x_j$, either
(i) $\lambda_{\min}(\nabla^2 f(x_j))\ge -\lambda_{\mathrm{nc}}$ and $\nu_{f,\sigma_jF}(x_j)>\epsilon_g$, or (ii) $\lambda_{\min}(\nabla^2 f(x_j))<-\lambda_{\mathrm{nc}}$.
In case (i), using Eq.~\eqref{eq:t-ARA-2nd-nu}, we obtain
\(f(x_j)-m_j(t_j)\ \ge\ \kappa^{-2}\omega(\kappa\nu_{f,\sigma_j F}(x_j))\ \geq\ \kappa^{-2}\omega(\kappa\epsilon_g)\).

In case (ii), set $z:=-\delta_j/(\sigma_j\Delta_j)$.
Since $\delta_j=\lambda_{\min}(\nabla^2 f(x_j))<-\lambda_{\mathrm{nc}}=-\sigma_j\sqrt{\epsilon_H}\Delta_j$, we have $z>\sqrt{\epsilon_H}$ and, for the step size $t_j$ in Eq.~\eqref{eq:t-negative-curvature-descent},
the explicit form of $m_j$ yields \(f(x_j)-m_j(t_j)\ge\tfrac{z^3}{6(\kappa+\kappa_F)^2}
\ \ge\ \tfrac{\epsilon_H^{1.5}}{6(\kappa+\kappa_F)^2}\).
Consequently, for all $j\in\mathcal{S}_k$,
\(f(x_j)-m_j(t_j)\ \ge\ \widetilde\epsilon
:=\min\!\left(\kappa^{-2}\omega(\kappa\epsilon_g),\ \tfrac{\epsilon_H^{1.5}}{6(\kappa+\kappa_F)^2}\right)\).
Plugging this into Eq.~\eqref{eq:telescoping-ARM-2nd} gives $f(x_0)-f_{\inf}\ \ge\ \eta_1\,\lvert \mathcal{S}_k\rvert\,\widetilde\epsilon$, and hence $\lvert \mathcal{S}_k\rvert\ \le\ \frac{f(x_0)-f_{\inf}}{\eta_1\,\widetilde\epsilon}=O(\widetilde\epsilon^{-1})$.

It remains to bound the number of rejected iterations $\lvert\mathcal{U}_k\rvert$.
For any $j\in \mathcal{S}_k$ we have $\sigma_{j+1}\ge \max(\sigma_{\min},\gamma_1\sigma_j)\ge \gamma_1\sigma_j$, while for any $j\in\mathcal{U}_k$ we have $\sigma_{j+1}\ge \gamma_2\sigma_j$.
Thus, letting $\sigma_{\max} = \max(\sigma_0,\gamma_3)$, we have $\sigma_0\,\gamma_1^{\lvert \mathcal{S}_k\rvert}\gamma_2^{\lvert \mathcal{U}_k\rvert}\ \le\ \sigma_k\ \le\ \sigma_{\max}$.
Taking logarithms and rearranging yields
\(\lvert \mathcal{U}_k\rvert \le \tfrac{\log(\sigma_{\max}/\sigma_0)-\lvert \mathcal{S}_k\rvert\log(\gamma_1)}{\log(\gamma_2)}\), and therefore $k=\lvert\mathcal{S}_k\rvert+\lvert\mathcal{U}_k\rvert = O(\lvert\mathcal{S}_k\rvert)=O(\widetilde\epsilon^{-1})$. Finally, since
\[
\widetilde\epsilon^{-1}
\le \kappa^{2}(\omega(\kappa\epsilon_g))^{-1} + \tfrac{6(\kappa+\kappa_F)^2}{\epsilon_H^{1.5}}
=O(\epsilon_g^{-2}+\epsilon_H^{-1.5}),
\]
the claimed iteration bound follows. \Halmos
\end{proof}

\begin{remark}
Suppose that Algorithm~2 generates a nonzero trial step, i.e., $t_j d_j \neq 0$. Then
\(\delta_j+\sigma_j\Delta_j>0\). Indeed, by the convention $\sqrt{x}=+\infty$ for $x<0$ and the rule in Line~6 of
Algorithm~2 that sets $t_j=0$ whenever the minimizer of $m_j$ is attained at $+\infty$
or $m_j$ is constant, a nonzero trial step can occur only when the model $m_j$ admits a
finite minimizer at some $t_j>0$, which excludes the cases
$\delta_j+\sigma_j\Delta_j\le 0$. Moreover, if Option~1 is used, or if Option~2 is used with $\lambda_{\min}(\nabla^2 f(x_j))\ge -\lambda_{\mathrm{nc}}$, then $\rho_j>0$. In the negative curvature branch of Option~2, i.e., Line 2 of Eq.~\eqref{eq:m-negative-curvature}, we only require $\rho_j\geq 0$, which is achieved by Eq.~\eqref{eq:d-negative-curvature}.
\end{remark}

\begin{remark}
ARM improves on
RNM in two
ways:
First,
ARM introduces a parameter $\sigma_j$ that controls the magnitude of regularization, mitigating the issue of excessively large regularization. Second,
ARM is effective when $f$ is $\ell F$-based $\kappa$-self-concordant, but we have no prior knowledge of $\ell$.
\end{remark}
\begin{remark}
The key advantage of
ARM over the trust region (TR) method and \cite[Algorithm~2.4.1]{Cartis-2022-Evaluation} is
that ARM computes
a better direction $d_j$ and
constructs a better model $m_j(\cdot)$ by leveraging the special self-concordance structure. Indeed, the TR method computes $d_j$ based on a local model $m_j(d) = f(x_j) + \nabla f(x_j)^\top d + \frac{1}{2}d^\top B_jd$, where $B_j$ is
a symmetric matrix
chosen
so that $f(x_j+d_j) < f(x_j)$.
The TR method either accepts or rejects the update $d_j$ based on the
reduction ratio
$r_j = \frac{f(x_j)-f(x_j + d_j)}{m_j(0)-m_j(d_j)}$, and always scales the TR radius based on $r_j$. In \cite[Algorithm~2.4.1]{Cartis-2022-Evaluation}, a local model $m_j(d) = f(x_j) + \nabla f(x_j)^\top d + \frac{1}{2}\sigma_j \|d\|^2$ is constructed and $d_j = \argmin_d m_j(d)$. All other steps are similar to
those
in
ARM. Another advantage of
ARM is its ability to handle the constrained case where $\dom(f) \neq \br^n$. It rejects
taking a step
when $x_j + d_j \notin \dom(f)$, while maintaining convergence guarantees.
\end{remark}

\section{Experiments} \label{sec:exp}

In this section we report the results of two different types of computational experiments. In the first set of tests we compared the performance of Algorithm~\ref{algorithm:ARM} (ARM) with
option 1 for solving nonnegative matrix factorization problems against that of four other algorithms.
In the second set of tests, we modified the well-known approximate natural gradient/Gauss-Newton method, KFAC~\cite{Martens-2015-Optimizing} so that it was transformed into an ``approximate Gauss-Newton'' version of Algorithm~\ref{algorithm:RNM} (RNM). We then compared its performance to that of the original KFAC algorithm in training four different convolutional-neural network/data-set combinations, to illustrate the use of our approach to nonconvex self-concordance to methods beyond those analyzed in this paper.

In our experiments, we found that RNM performs well for training convolutional neural networks (CNNs) but is less effective for nonnegative matrix factorization (NMF),
which can be explained by the different magnitudes of $F$
in these two cases. In our CNN experiments, $F$ is a quadratic with
coefficients on the order of $10^{-2}$, which is consistent with the asymmetric Hessian spectra reported in NN
training~\cite{Zhang-2024-Transformers}, where negative curvature is mild. In these settings, a small regularizer is sufficient to stabilize the updates and capture local curvature. In contrast, the objective function $f$ and its Hessian entries are much larger in scale for NMF. Near a local minimizer, applying the same globally chosen $F$ introduces excessive regularization, which distorts the update direction and ultimately degrades performance.
\paragraph{Nonnegative matrix factorization (NMF)}
Let
$\UCal[a,b]$
denote
the uniform distribution on $[a,b]$. For the KL divergence formulation, we generate the data matrix $Z = \hat{X}\hat{Y} + 0.01\hat{Z}$, where $\hat{X} \in \br^{m \times r}$, $\hat{Y} \in \br^{r \times n}$, $\hat{Z} \in \br^{m \times n}$ and their entries are
sampled i.i.d. from $\UCal[0,1]$. For the MSE loss, we generate $M = \hat{X}\hat{Y}$, where $\hat{X} \in \br^{m \times r}$, $\hat{Y} \in \br^{r \times n}$ and their entries are sampled i.i.d. from $\UCal[0,1]$. We compute the singular value decomposition of $M = U\diag(\sigma_1, \dots, \sigma_r, 0, \dots, 0)V^\top$ and construct the data matrix $Z = U\mathrm{diag}(\sigma_1, \dots, \sigma_r, \sigma_{r+1}, \dots, \sigma_{\min(m,n)})V^\top$, where $\sigma_{r+1}, \dots, \sigma_{\min(m,n)}$ are sampled i.i.d. from $\mathcal{U}[0, 0.1\sigma_r]$. The optimal function value is given by $\frac{1}{2mn}\|Z - \hat{X}\hat{Y}\|_F^2$. Thus, the optimality gap is defined as the difference between the current function value and the optimal function value. Throughout, we set $m=100$, $n=20$, and $r=10$. We consider the case where the initialization of $X$ and $Y$ are independent, with each entry drawn i.i.d. from $\UCal[0,1]$.

We test the performance of Algorithm~\ref{algorithm:ARM} on NMF with both MSE loss and KL divergence. For the MSE loss (Eq.~\eqref{prob:NMF_1}), we set the base function to $F_{\text{MSE}}(X,Y)=(\|X\|_F^2+\|Y\|_F^2+1)^2 - \sum_{i, k} \log(X_{ik}) -\sum_{k, j} \log(Y_{kj})$, while for the KL divergence (Eq.~\eqref{prob:NMF_2}), we set it to $F_{\text{KL}}(X,Y)= - \frac{1}{mn} (\sum_{i, k} \log(X_{ik}) + \sum_{k, j} \log(Y_{kj}))$. By Section~\ref{sec:prelims}, $f_{\text{MSE}}$ in Eq.~\eqref{prob:NMF_1} is $\ell_1 F_{\text{MSE}}$-based self-concordant for some $\ell_1>0$, and $f_{\text{KL}}$ in Eq.~\eqref{prob:NMF_2} is $\ell_2 F_{\text{KL}}$-based self-concordant for some $\ell_2>0$. We implement Algorithm~\ref{algorithm:ARM} with Option~1 and the direction $d_j = -(\nabla^2(f+\sigma_j F)(x_j))^{+} \nabla f(x_j)$; see also Lemma~\ref{lem:ARNM}, Eq.~\eqref{eq:d-ARNM}, and Eq.~\eqref{eq:m-ARNM}.\\
We compare the performance of Algorithm~\ref{algorithm:ARM} against \texttt{IPOPT} (from the python package cyipopt) \cite{Wachter-2006-Implementation}, \texttt{L-BFGS-B}, and \texttt{trust-constr} methods from \texttt{scipy}~\cite{Virtanen-2020-Scipy}. We also implement \texttt{COCARC}~\cite{Cartis-2012-Adaptive}, which applies the framework of adaptive cubic regularization \cite{Cartis-2011-Adaptive, Cartis-2011-Adaptive2} to non-convex optimization with convex constraints, to illustrate the advantages of our method over cubic regularization methods.
In both Algorithm~\ref{algorithm:ARM} and \texttt{COCARC}, we
set
$\sigma_0=1$, $\eta_1=0.01$, $\eta_2=0.9$, $\gamma_1=0.5$, and $\gamma_2=\gamma_3=2$, following \cite[Section~2.4]{Cartis-2022-Evaluation}. In Algorithm~\ref{algorithm:ARM}, we set $\kappa=1$.

As shown in Figure~\ref{fig:main}, Algorithm~\ref{algorithm:ARM} (ARM) consistently outperforms the other methods
to which it is compared on both the MSE (Frobenius norm) (Eq.~\eqref{prob:NMF_1}) and the KL divergence (Eq.~\eqref{prob:NMF_2}) formulations of NMF.
Regarding
the theoretical convergence guarantees of the compared methods in these NMF settings and more generally, we note the following. (i) Since the NMF objective has unbounded sublevel sets and unbounded second- and third-order derivatives, \texttt{COCARC}~\cite{Cartis-2012-Adaptive} is not guaranteed to converge;
and its convergence analysis relies on the Lipschitz constant of the Hessian, which is generally difficult to estimate in practice and can lead to conservative step-size choices. (ii) The convergence of \texttt{L-BFGS-B} on non-convex problems is more subtle: while BFGS enjoys global convergence and superlinear local convergence for convex objectives~\cite{Wright-2006-Numerical}, it is not guaranteed to converge to a stationary point in the nonconvex setting, and counterexamples demonstrating cycling behavior are known~\cite{Mascarenhas-2004-bfgs, Dai-2002-Convergence}; positive results exist, to the best of our knowledge, only under modified update rules such as the cautious BFGS strategy of~\cite{Li-2001-Global}. (iii) For NMF, ARM requires no knowledge of any Lipschitz constant and adaptively tunes the regularization parameter $\sigma_j$, while still enjoying the convergence guarantees established in Theorem~\ref{thm:ARM-1st}.

\begin{figure}[!t]
\centering
\begin{minipage}{0.48\textwidth}
\includegraphics[width=\linewidth]{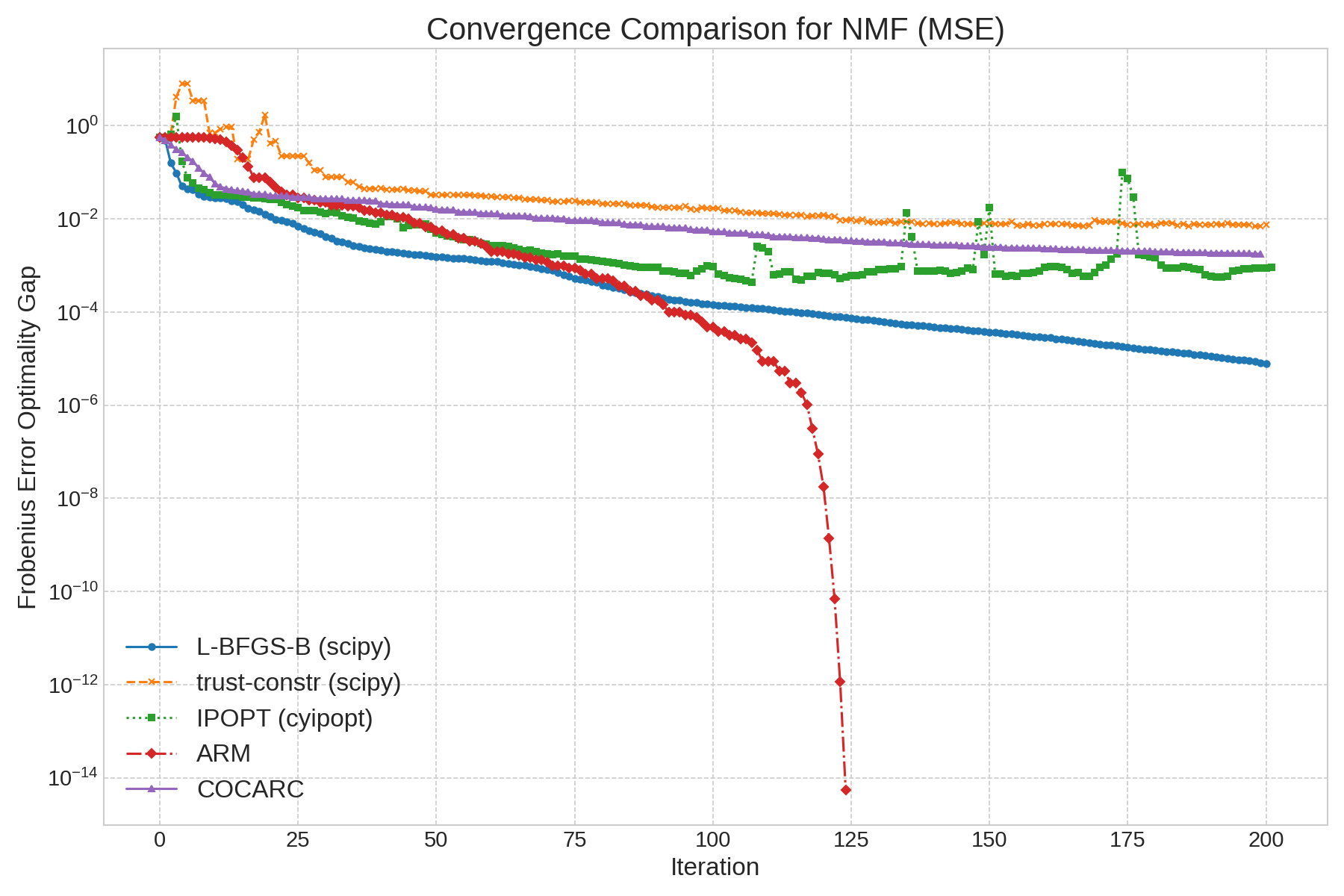}
\end{minipage}
\hfill
\begin{minipage}{0.48\textwidth}
\includegraphics[width=\linewidth]{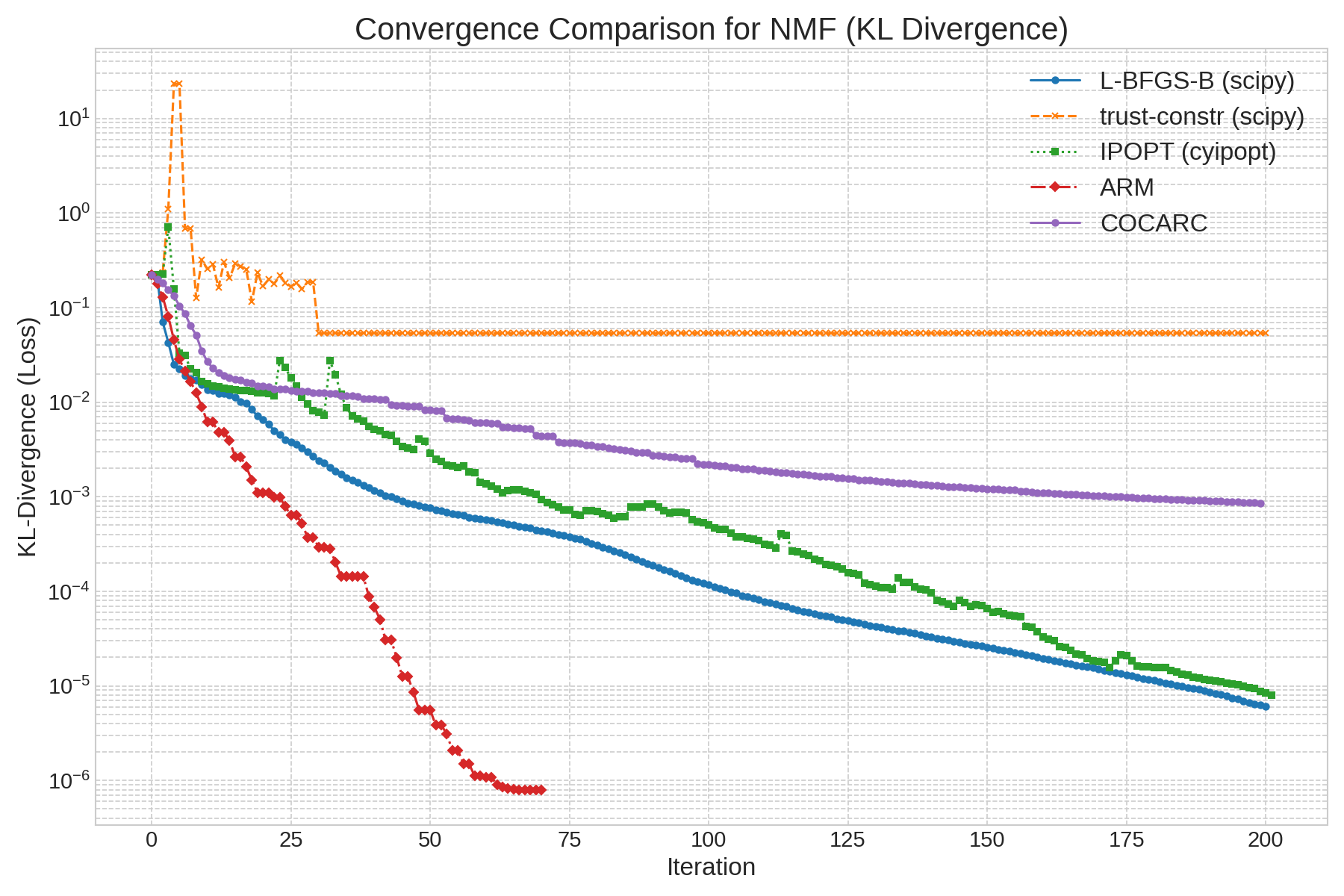}
\end{minipage}
\caption{Performance comparison on NMF with MSE loss (left) and KL divergence (right).} \label{fig:main}  \vspace*{-1.5em}
\end{figure}

\paragraph{Training convolutional neural networks (CNNs)} The use of Newton's methods in large-scale 
NN
training is prohibitively expensive, as computing and inverting the Hessian is costly. To overcome this limitation, the Kronecker-Factored Approximate Curvature (KFAC) method was introduced in~\cite{Martens-2015-Optimizing} as a scalable alternative to Newton's method and natural gradient descent.
Specifically, KFAC is an approximate natural gradient method that employs a Kronecker-factorized version of a positive
semi-definite Fisher Information Matrix (FIM), which is equivalent to the generalized Gauss-Newton (GGN) matrix in certain cases \cite{Martens-2012-Training, Pascanu-2014-Revisiting, Martens-2020-New, Abreu-2025-Potential}, to approximate the Hessian matrix in a regularized Newton method. It avoids the extra work needed
to compute the exact Hessian and
further greatly reduces the memory needed by approximating the FIM by
a block-diagonal matrix based on the layer-wise structure of
NNs.
We consider a supervised learning problem with data distribution $(x, y) \sim \DCal$, where $x$ denotes the input and $y$ the target. A feed-forward
NN
defines a parametric function $f_\theta : \br^{d_x} \to \br^{d_y}$, with $\theta$ consisting of weight matrices and biases across all layers. In particular,
in
the $l$-th layer,
$s_l = W_l a_l + b_l$, $h_l = \phi(s_l)$, where $a_l \in \br^{d_{\textnormal{in}}}$ 
is the vector of inputs,
$W_l \in \br^{d_{\textnormal{out}} \times d_{\textnormal{in}}}$ and $b_l \in \br^{d_{\textnormal{out}}}$ are the parameters,
and
$s_l \in \br^{d_{\textnormal{out}}}$
and $h_l \in \br^{d_{\textnormal{out}}}$
are, respectively, the vectors of pre-activations
and outputs (after element-wise activation $\phi$).
\begin{algorithm} \small
\caption{KFAC \textbf{(i)} \cite{Martens-2015-Optimizing, Zhang-2019-Three} and KFAC with weak self-concordance (KFAC-WSC) \textbf{(ii)}.}\label{algo:KFAC-WSC}
\vskip6pt
\begin{algorithmic}[1] \small
\State \textbf{Input:} $\eta_k$: learning rate; $\omega$: weight decay; $\gamma$: damping; $\mu$: momentum, default 0.9.
\State \textbf{Input:} If choosing \textbf{(ii)}, input $\kappa$: weak self-concordance parameter, default 1; input $\beta$: layerwise learning rate momentum, default 0.99.
\State \textbf{Initialization:} $k \gets 0$ and initialize $\{W_l\}_{l=0}^{L-1}$.
\While{stopping criterion not met}
\State $k \gets k+1$.
\State Sample a mini-batch and update $\{S_l\}_{l=0}^{L-1}$, $\{A_l\}_{l=0}^{L-1}$ with moving average.
\State Compute the inverses $\{S_l^{-1}\}_{l=0}^{L-1}$, $\{A_l^{-1}\}_{l=0}^{L-1}$.
\For{$l=0,\dots,L-1$}
\State Sample a mini-batch $B$.
\State $G_l \gets \nabla_{W_l} \mathcal{L}_B(W)$, $D_l \gets -(A_l+\gamma I)^{-1} G_l (S_l+\gamma I)^{-1}$, $M_l \gets \mu M_l + (1-\mu) D_l$.
\State Compute $\lambda_l$ according to one of two options: \textbf{(i)} $\lambda_l\gets0$. \textbf{(ii)} $\lambda_l^2 \gets \beta\lambda_l^2 - (1-\beta) \mathrm{vec}(G_l)^\top\mathrm{vec}(D_l)$.
\State $W_l \gets W_l + \frac{\eta_k}{1+\kappa \lambda_l} M_l  -\eta_k\omega W_l$.
\EndFor
\EndWhile
\end{algorithmic}
\end{algorithm}

The network is trained by minimizing $\LCal(\theta) = \EE_{(x,y) \sim \DCal}[L(f_\theta(x), y)]$, where $L$ is a loss function, and the FIM is $P(\theta) = \EE_{(x,y) \sim \DCal}[ \nabla_\theta \log p_\theta(y \mid x) (\nabla_\theta \log p_\theta(y \mid x))^\top]$ where $p_\theta(y \mid x)$ denotes the model's predictive distribution. For the weight matrix $W_l$, the gradient of $\LCal$ with respect to $W_l$ can be written in terms of the input activation and the back-propagated signal. In particular, $\nabla_{W_l} L = g_l a_l^\top$, where $g_l = \nabla_{s_l} L \in \br^{d_{\textnormal{out}}}$ is the gradient of $L$ with respect to the pre-activations $s$. The Fisher block corresponding to $W_l$, equivalently, the diagonal block of the Gauss–Newton matrix, is given by $P_l = \EE[\operatorname{vec}(\nabla_{W_l} L) \operatorname{vec}(\nabla_{W_l} L)^\top] = \EE[(a_l\otimes g_l)(a_l \otimes g_l)^\top]$. The direct computation of $P_l$ is intractable, as it scales with the square of the number of parameters in $W$. The key idea behind KFAC is to treat $a_l$ and $g_l$ as independent, yielding $P_l \approx \EE[a_l a_l^\top] \otimes \EE[g_lg_l^\top]$. For any $X \in \br^{d_{\textnormal{out}} \times d_{\textnormal{in}}}$, we have $P^{-1}_l \mathrm{vec}(X) \approx (\EE[a_la_l^\top])^{-1} X (\EE[g_lg_l^\top])^{-1} = A_l^{-1}XS_l^{-1}$,
where $A_l = \EE[a_la_l^\top]$ and $S_l = \EE[g_lg_l^\top]$. Consequently, denoting the partial gradient of $\mathcal{L}$ with respect to $W_l$ as $\nabla_{W_l}\mathcal{L}(W) \in \mathbb{R}^{d_{\mathrm{out}} \times d_{\mathrm{in}}}$, the natural gradient (i.e., the Gauss--Newton update) is approximated by $A_l^{-1}\nabla_{W_l}\mathcal{L}(W)\,S_l^{-1}$.

In practice, it is common to add regularization to $A_l$ and $S_l$ to improve both generalization and numerical stability \cite{Martens-2015-Optimizing}. The Gauss-Newton update with Tikhonov regularization \cite{Martens-2015-Optimizing} is $D_l = -(A_l + \gamma I)^{-1} \nabla_{W_l} \mathcal{L}(W) (S_l + \gamma I)^{-1}$.
We interpret $D_l$ as an instance of the regularized Newton direction $d_j = -(\nabla^2(f + F)(x_j))^{-1}\nabla f(x_j)$ from Algorithm~\ref{algorithm:RNM}, where the Hessian regularization term $\nabla^2 F(x_j)$ is replaced by the scalar shift $\gamma I$, and the damping parameter $\gamma$ plays the role of the weak self-concordance parameter $\ell$. Accordingly, we identify $\lambda_l = \sqrt{-\operatorname{vec}(D_l)^\top \operatorname{vec}(\nabla_{W_l}\mathcal{L}(W))}$ as a layer-wise approximation of the quantity $\lambda_j = \sqrt{-\nabla f(x_j)^\top d_j}$ computed in Line~4 of Algorithm~\ref{algorithm:RNM}---namely, the square root of the inner product between the gradient and the negative update direction, which serves as the regularized Newton decrement. Our KFAC-WSC implementation (Algorithm~\ref{algo:KFAC-WSC}) therefore shares the same underlying structure as Algorithm~\ref{algorithm:RNM} and the classical damped Newton method for convex self-concordant functions~\cite{Nesterov-2018-Lectures}:
the main difference over regular Gauss-Newton's method is that the
step size applied to the (regularized) Newton update is taken proportional to $\frac{1}{1 + \kappa\lambda_l}$, where $\lambda_l$ is the square root of the regularized Newton decrement, as prescribed by Line~5 of Algorithm~\ref{algorithm:RNM}. Consequently, the step sizes
in KFAC-WSC are different for different layers.

We follow the techniques of \cite{Zhang-2019-Three} and evaluate our algorithms by training VGG16 \cite{Simonyan-2015-Very} and ResNet32 \cite{He-2016-Deep, Zagoruyko-2016-Wide} on CIFAR-10 and CIFAR-100 \cite{Krizhevsky-2009-Learning}. The test accuracies of SGD and Adam \cite{Kingma-2015-Adam} are copied from \cite{Zhang-2019-Three}, where a grid search was conducted to find the best hyperparameters. For KFAC \cite{Martens-2015-Optimizing} and KFAC-WSC (our method), we conducted our own experiments on a single NVIDIA V100 SXM2 GPU from San Diego Supercomputer Center \cite{Strande-2021-Expanse}. 
The results are summarized in Table~\ref{tab:cnn-experiments}. Both KFAC and KFAC-WSC were trained for 100 epochs, and the learning rate was decayed by a factor of $0.1$ at the 40th and 80th epochs. We performed grid searches to identify the hyperparameters that maximize test accuracy during training. For KFAC, we adopted the best hyperparameters from \cite{Zhang-2019-Three}, on top of which we tuned the learning rate over $\{3\times 10^{-4}, 1\times 10^{-3}, 3\times 10^{-3}\}$, selecting the best learning rate based on validation accuracy. Our results match \cite{Zhang-2019-Three} closely. For KFAC-WSC, the initial learning rate was selected from $\eta_0 \in \{10^{-3}, 3\times 10^{-3}, 10^{-2}\}$. The weight decay factor was chosen from $\omega \in \{10^{-2}, 3\times 10^{-2}, 10^{-1}, 3\times 10^{-1}, 1.0\}$, and the damping factor from $\gamma \in \{3\times 10^{-3}, 10^{-2}, 3\times 10^{-2}\}$. This setup matches the $3\times5\times3$ grid used in \cite{wang-kfac-pytorch}, a PyTorch implementation of \cite{Zhang-2019-Three}. In practice, using the same hyperparameters as KFAC for KFAC-WSC leads to degraded performance. This is evident since we are using smaller learning rates for each layer. Our method improved training efficiency on VGG16 for CIFAR-10 and CIFAR-100, while being outperformed by KFAC on ResNet32.


A convergence analysis of KFAC \cite{Martens-2015-Optimizing} and, its extension, KFAC-WSC is beyond the scope of this paper. Rather, the goal of this experiment is to empirically assess the effectiveness of incorporating the self-concordant step size rule from Algorithm~\ref{algorithm:RNM} into a practical second-order optimizer.

\begin{table}[!t]
\centering
\caption{Test accuracy of different neural network optimizers. We replicate the KFAC experiments from \cite{Zhang-2019-Three} using 10 random seeds and copied the results of SGD and Adam (\cite{Kingma-2015-Adam,Loshchilov-2017-Decoupled}) from \cite{Zhang-2019-Three}. Values following ``$\pm$'' denote standard deviations across seeds. $(\eta_0^*,\omega^*,\gamma^*)$ are the best used initial learning rate, weight decay, and damping factor respectively selected from grid searches.} \label{tab:cnn-experiments}
\begin{tabular}{|c|c||c|c||cc||cc|}
\hline
 &  & SGD & Adam &
\multicolumn{2}{c||}{KFAC} &
\multicolumn{2}{c|}{KFAC-WSC} \\ \hline
 Dataset & Model & Acc. & Acc. & Acc. & $(\eta_0^*,\omega^*,\gamma^*)$ & Acc. & $(\eta_0^*,\omega^*,\gamma^*)$ \\ \hline
CIFAR10 & VGG16 		& 93.39 	& 93.62 	& 93.88{\tiny$\pm$0.16} & {\tiny$(1\times 10^{-3},0.3,1\times 10^{-3})$}& \textbf{94.13{\tiny$\pm$0.17}} & {\tiny$(3\times 10^{-3},3\times 10^{-2},3\times 10^{-2})$}\\ \hline
CIFAR10 & ResNet32 	& 95.14 	& 94.66 	& \textbf{95.24{\tiny$\pm$0.13}}&{\tiny$(3\times 10^{-4},1,1\times 10^{-3})$}& 95.07{\tiny$\pm$0.14} & {\tiny$(3\times 10^{-3},3\times 10^{-2},3\times 10^{-2})$} \\ \hline
CIFAR100 & VGG16  		& 73.31 	& 74.22 	& 73.53{\tiny$\pm$0.47}& {\tiny$(1\times 10^{-3},1,1\times10^{-3})$
} & \textbf{74.27{\tiny$\pm$0.19}} & {\tiny$(3\times 10^{-3},3\times 10^{-2},3\times 10^{-2})$}\\ \hline
CIFAR100 & ResNet32 	& 77.73 	& 77.40 	& \textbf{78.08{\tiny$\pm$0.21}} & {\tiny$(3\times 10^{-4},1,1\times 10^{-3})$}& 77.07{\tiny$\pm$0.39} & {\tiny$(3\times 10^{-3},3\times 10^{-2},3\times 10^{-3})$} \\ \hline
\end{tabular} \vspace*{-1.5em}
\end{table}

\section{Conclusions}\label{sec:conclu}
We generalized the notion of self-concordance to non-convex settings by introducing weakly self-concordant functions and $F$-based self-concordant functions. Based on these two classes, we proposed regularized Newton and adaptive regularization methods and proved their global and local convergence guarantees. We also conducted a few preliminary numerical experiments to evaluate their effectiveness. Indeed, our adaptive regularization method outperformed adaptive cubic regularization, L-BFGS, and trust-region methods on non-negative matrix factorization problems in terms of both MSE and KL divergence losses. Our KFAC-based regularized Newton method improved training efficiency on VGG16 for CIFAR-10 and CIFAR-100, while being outperformed by KFAC on ResNet32. In summary, we believe that the analytic and experimental results that we presented illustrate the potential
of extending the powerful properties of self-concordance
to non-convex optimization.

\section*{Acknowledgments} The second author is partially supported by a start-up fund at the University of Hong Kong and by the Hong Kong Research Grants Council under ECS project 27301425. The third author is partially supported by a start-up fund at Columbia University. We also thank the support of GPU hours from NSF ACCESS \cite{Boerner-2023-Access}.

\appendix
\section{Additional Results on Weak Self-Concordant Functions}\label{sec:appendix}
In Section~\ref{subsec:convergence-to-a-proximal-stationary-point}, we present an alternative way to measure stationarity for weakly self-concordant functions. In Section~\ref{subsec:IPPM}, we introduce an inexact proximal point method for minimizing weakly self-concordant functions that relies on Hessian–vector products (HVPs) and achieves an iteration complexity comparable to that of RNM and ARM, with high probability. In Section~\ref{subsec:appendix-missing-proofs}, we prove our convergence results.

For a self-concordant function $f:\mathbb{R}^n\to\mathbb{R}\cup\{+\infty\}$ and $x\in\dom(f)$, we define $\lambda_f(x) = \sqrt{\nabla f(x)^\top (\nabla^2f(x))^{-1}\nabla f(x)}$. For a matrix $A \succ 0$, we denote its condition number by $\mathrm{cond}(A)$ and define $\|x\|_A=\sqrt{x^\top Ax}$. We also define the following absolute constants $\alpha_\star = 0.0001$, $R_1 = 0.49$, $R_2 = \frac{1}{\sqrt{1 - \alpha_\star}} R_1$, $R_3 = \frac{\sqrt{1 - \alpha_\star}}{1 + \alpha_\star} R_1$, $C_1 = 9$, $C_2 = 0.95$, and $C_3 = \frac{R_2^{-1} - 1}{R_2^{-1} - 2}$.

\subsection{Convergence to a Proximal Stationary Point}
\label{subsec:convergence-to-a-proximal-stationary-point}
For any $\eta>0$, we define the 
Moreau envelope $f_{\eta}:\mathbb R^n \to \mathbb R\cup \{+\infty\}$ \cite{Moreau-1965-Proximite} of $f$ by $f_{\eta}(y) = \inf_{x\in \mathbb{R}^n}\{f(x) + \tfrac{1}{2\eta}\|y-x\|^2 \}$. The gradient norms of the Moreau envelope, $\|\nabla f_{\eta}(x)\|$, have been used as a measure of stationarity for weakly convex functions \cite{Davis-2019-Stochastic}. For weakly self-concordant functions, we show next that RNM (see Algorithm~\ref{algorithm:RNM}) and Algorithm~\ref{algorithm:ARM} can find a point with $\|\nabla f_{\eta}(x)\| \le \epsilon$ (for any small enough $\eta$) in $O(\epsilon^{-2})$ iterations. 
\begin{lemma}
\label{lem:Moreau-envelope-condition}
Suppose that $f:\mathbb{R}^n\to\mathbb{R}\cup\{+\infty\}$ is $(\kappa,\ell)$-weakly self-concordant and lower-bounded. For any $\mu>\ell$, $\epsilon>0$, and $x\in\dom(f)$, if
$\nu_{f,\frac12\mu\|\cdot\|^2}(x)\leq \frac{\sqrt{\mu-\ell}\epsilon}{\mu + \kappa\sqrt{\mu-\ell}\epsilon}$, then we have $\|\nabla f_{1/\mu}(x)\| \leq \epsilon$.
\end{lemma}
\begin{proof}{Proof}
Let $\epsilon'=\frac{\sqrt{\mu-\ell}\epsilon}{\mu + \kappa\sqrt{\mu-\ell}\epsilon}$ and $\varphi(y) = f(y) + \frac{1}{2}\mu \|y-x\|^2$. Thus, we assume $\nu_{f,\frac12\mu\|\cdot\|^2}(x)=\lambda_{\varphi}(x) \leq \epsilon'$. It follows from \cite[Lemma 2.2]{Davis-2019-Stochastic} that $\tfrac1{\mu}{\|\nabla f_{1/\mu}(x)\|} = \|x - \arg\min \varphi\|$. Moreover, $\nabla^2\varphi\geq (\mu-\ell) I$ implies that $\sqrt{\mu - \ell} \|x - \arg\min \varphi\|\leq \|x - \arg\min \varphi\|_{\nabla^2 \varphi(x)}$. We define $\omega_\star'(t) = \frac{t}{1 - t}$ for $t\in[0,1)$ and apply \cite[Eq.~(5.2.4)]{Nesterov-2018-Lectures} with $f = \varphi$. In the notation of \cite{Nesterov-2018-Lectures}, we note that $x_f^\star = \arg\min \varphi$, $M_f = \kappa$, $\|\cdot\|_x = \|\cdot\|_{\nabla^2\varphi(x)}$. Thus, we have $\|x - \arg\min_y \varphi(y)\|_{\nabla^2 \varphi(x)} \leq \kappa ^{-1}\omega_\star'(\kappa \lambda_{\varphi}(x))$. Therefore, we obtain $\frac{\sqrt{\mu-\ell}}{\mu} \|\nabla f_{1/\mu}(x)\|\leq\|x - \arg\min \varphi\|_{\nabla^2 \varphi(x)} \leq \kappa ^{-1}\omega_\star'(\kappa \lambda_{\varphi}(x))\leq \kappa^{-1} \omega_\star'(\kappa \epsilon') = \tfrac{\sqrt{\mu - \ell}}{\mu} \epsilon$. \Halmos
\end{proof}

\subsection{Inexact Proximal Point Method (IPPM)}
\label{subsec:IPPM}

We develop an algorithm for minimizing a $(\kappa,\ell)$-weakly self-concordant function $f$ using only gradient and Hessian–vector product (HVP) oracles. Our approach is based on the proximal point method \cite{rockafellar1976monotone}. Given the current iterate \( z_j \in \dom(f)\), we consider the regularized objective \( f_j : \mathbb{R}^n \to \mathbb{R}\cup\{+\infty\} \) defined as \( f_j(y) = f(y) + \frac{1}{2} \mu \| y - z_j \|^2 \). The update of the proximal point method is given by \( z_{j+1} = \arg\min  f_j \). When \( \mu > \ell \), each function \( f_j \) is self-concordant.

To solve these subproblems inexactly, we introduce Algorithm~\ref{algorithm:Newton-CG}. The algorithm combines the Newton–CG (conjugate gradient) method \cite[Algorithm~7.1]{Wright-2006-Numerical} with the Lanczos method \cite{Kuczynski-1992-Estimating}. By exploiting self-concordance, it avoids line searches, which is its main advantage over traditional Newton–CG. For a self-concordant function \( f \), the algorithm outputs, with probability $1-\delta$, \( y\in \dom(f) \) such that either \( \lambda_f(y) \leq \epsilon_1 \) and \( \lambda_f(x_0) \geq \sqrt{\frac{1 - \alpha_\star}{1 + \alpha_\star}} \epsilon_0 \), or \( y = x_0 \) and \( \lambda_f(x_0) \leq \epsilon_0 \). A more precise convergence theorem, together with its proof, is deferred to Theorem~\ref{thm:convex-SC-newton-CG} in Appendix~\ref{subsec:appendix-missing-proofs}. This algorithm may be of independent interest in convex optimization.

\begin{algorithm}
\caption{A Newton-CG Method for Self-Concordant Functions}
\label{algorithm:Newton-CG}
\vskip6pt
\begin{algorithmic}[1] \small
\State \textbf{Input:} $\kappa\geq0$, $\epsilon_0>0$, $\epsilon_1>0$, $\beta>1$, $\delta\in[0,1)$.
\State \textbf{Initialization:} $x_0\in \dom(f)$.
\State $\beta_{0} \gets \text{sqrt-cond}(\nabla^2f(x_{0}),0.5\delta,\beta)$, $\texttt{flag}\gets \texttt{True}$.
\For{$k = 0, 1, 2, \dots$}
    \State $g_k\gets \nabla f(x_k)$, $h_k \gets$ CG-Inverse$(\nabla^2 f(x_k), -g_k, \beta_k,\alpha_\star)$.
    \State $\rho_k \gets -h_k^\top g_k$, $\delta_k \gets \nabla^2f(x_k)[h_k,h_k]$.
    \Return $x_k$ \textbf{if} $\rho_k \leq (1-\alpha_\star)\epsilon_1^2$ \textbf{or} ($k=0$ \textbf{and} $\rho_0 \leq (1-\alpha_\star)\epsilon_0^2$).
    \If{$\rho_k > \frac{R_1^2}{\kappa^2}$}
    \State $t_k \gets \frac{\rho_k}{\delta_k+\kappa \rho_k\sqrt{\delta_k}}$, $B_k\gets (1+ \frac{\sqrt{1+\alpha_\star}}{1-\alpha_\star} \kappa \sqrt{\rho_k})^2$, $\beta_{k+1} \gets B_k\beta_k$.
    \Else
    \If{\texttt{flag}}
    \State $\beta^\star \gets \text{sqrt-cond}(\nabla^2f(x_{k}),0.5\delta,\beta_{k})$, $\beta_{\mathrm{local}} \gets C_3^4 \beta^\star$, $\texttt{flag}\gets \texttt{False}$.
    \EndIf
    \State $t_k \gets \frac{\rho_k}{\delta_k+2\kappa \rho_k\sqrt{\delta_k}}$, $\beta_{k+1}\gets \beta_{\mathrm{local}}$.
    \EndIf
    \State $x_{k+1}\gets x_k + t_kh_k$.
\EndFor
\State \textbf{function} sqrt-cond($H,\delta,\beta$):
\State \quad Use the Lanczos algorithm to obtain $\lambda_1,\lambda_2$ such that $|\lambda_1-\lambda_{\max}(H)|\leq \frac{1}{3}\lambda_{\max}(H)$ and $|\lambda_2-\lambda_{\min}(H)|\leq \frac{1}{3}\lambda_{\min}(H)$. \textbf{return} $\sqrt{2\lambda_1/\lambda_2}$.
\State \textbf{function} CG-Inverse($H,g,\beta,\alpha$):
\State \quad Perform $\min\left\{n,\left\lfloor\log_{\tfrac{\beta-1}{\beta+1}} {0.5\alpha}\right\rfloor + 1\right\}$ CG iterations to compute $H^{-1}g$ start from $0$.
\end{algorithmic}
\end{algorithm}

We are now ready to introduce Algorithm~\ref{algo:IPPM} (IPPM). At each iterate \( z_j \), IPPM computes an inexact proximal point \( z_{j+1} \) by approximately solving the subproblem \(\min f_j \) to a prescribed accuracy using Algorithm~\ref{algorithm:Newton-CG}. The convergence of IPPM is established in the Theorem \ref{thm:IPPM}, whose proof is deferred to the end of Appendix~\ref{subsec:appendix-missing-proofs}. This result relies on Assumption~\ref{assumption:proximal-point-condition-number}, which requires that all proximal subproblems are sufficiently well conditioned.

\begin{algorithm}
\caption{Inexact Proximal Point Method (IPPM)}
\label{algo:IPPM}
\vskip6pt
\begin{algorithmic}[1] \small
\State \textbf{Input:} $\kappa\geq 0$, $\mu>0$, $\epsilon>0$, $\beta>1$, $\delta\in[0,1)$, $\Delta>0$.
\State \textbf{Initialization:} $z_0\in\dom(f)$.
\State $K\gets\frac{8\mu\Delta}{(\mu - \ell)\epsilon^2}$, $\delta'\gets \left(K+2\right)^{-1}\delta$, $\beta' \gets C_3^2\beta$, $\epsilon' \gets \left(\sqrt{\frac{1-\alpha_\star}{1+\alpha_\star}}-0.5\right)\epsilon$
\For{$j = 0, 1, 2, \dots$}
\State Set $f_j:\mathbb{R}^n\to\mathbb{R}\cup\{+\infty\}$ with $f_j(y):= f(y) + \frac{1}{2}\mu\|y-z_j\|^2$.
\State $z_{j+1}\gets$ The output of Algorithm~\ref{algorithm:Newton-CG} with objective function $f_j$, input $\kappa\gets\kappa$, $\epsilon_0\gets\epsilon$, $\epsilon_1\gets\epsilon'$, $\beta\gets\beta'$, $\delta\gets\delta'$, and initialization $x_0\gets z_j$.
\Return $z_j$ \textbf{if} $z_{j+1}=z_j$.
\EndFor
\end{algorithmic}
\end{algorithm}

\begin{assumption}
\label{assumption:proximal-point-condition-number}
In IPPM, there exists $B>0$ such that $\mathrm{cond}(\nabla^2 f(z_0)+\mu I)\leq B^2$ and we have $\mathrm{cond}(\nabla^2f_j(\arg\min f_j)) \leq B^2$ for all $j\geq 0$.
\end{assumption}
\begin{theorem}
\label{thm:IPPM}
Assume that $f$ is $(\kappa,\ell)$-weakly self-concordant, lower bounded, and that Assumption~\ref{assumption:proximal-point-condition-number} holds. Let $\mu > \ell$, $\beta \geq B$, $\Delta \geq f(z_0) - \inf_y f(y)$, and $\epsilon' \leq \tfrac{R_2}{\kappa}$. Then, with probability at least $1 - \delta$, IPPM terminates in at most $\left\lfloor \tfrac{8\mu(f(z_0)-\inf_y f(y))}{(\mu-\ell)\epsilon^2} \right\rfloor + 2$ iterations, and its output $y$ satisfies $\nu_{f,\tfrac{1}{2}\mu \|\cdot\|^2}(y) \leq \epsilon$. Moreover, the total number of HVP computations performed by IPPM is $O\!\left(\epsilon^{-2}(\log(\epsilon^{-1}) + \log(\delta^{-1}))\right)$.
\end{theorem}

\subsection{Convergence Analysis}\label{subsec:appendix-missing-proofs}
In this section, we prove the convergence of Algorithms \ref{algorithm:Newton-CG} and \ref{algo:IPPM}. To do so, we first present some technical lemmas. Lemma~\ref{lem:lanczos} analyzes the function sqrt-cond in Algorithm~\ref{algorithm:Newton-CG} by a classical result \cite{Kuczynski-1992-Estimating} on the Lanczos method. Lemma~\ref{lem:cg-apx} analyzes the function CG-Inverse in Algorithm~\ref{algorithm:Newton-CG} using \cite[Theorem 5.5]{Wright-2006-Numerical}. Lemma~\ref{lem:condition-numbers-of-two-points-near-optimal} describes the relationship between the condition numbers of the Hessians at two nearby points for self-concordant functions.

\begin{lemma}
\label{lem:lanczos}
Let $A\succ 0$. Given $\delta \in [0,1)$ and $\beta \geq \sqrt{\mathrm{cond}(A)}$, sqrt-cond($A, \delta, \beta$) outputs a number in $[\sqrt{\mathrm{cond}(A)},2\sqrt{\mathrm{cond}(A)}]$ with probability at least $1 - \delta$. The total number of matrix-vector multiplications with $A$ required by the method is at most $C_L \min\{n,\, \beta \log(n\delta^{-2})\}$, where $C_L > 0$ is a constant.
\end{lemma}
\begin{proof}{Proof}
Let $\epsilon'=\frac{1}{10\beta^2}$. It follows from \cite[Theorem 4.2]{Kuczynski-1992-Estimating} that, with probability at least $1-\frac{1}{2}\delta$, the Lanczos method outputs $\lambda_1$ such that $|\lambda_1-\lambda_{\max}(A)|\leq \frac{1}{3}\lambda_{\max}(A)$ in $O(\min\{n,\, \log(n\delta^{-2})\})$ matrix-vector multiplications. With probability at least $1 - \frac{1}{2}\delta$, the Lanczos method outputs $\lambda'$ such that $\left| \lambda' - \lambda_{\max} \left( \tfrac{2\lambda_1}{1 - \epsilon'} I - A \right) \right|\leq \epsilon' \cdot\allowbreak\lambda_{\max} \left( \tfrac{2\lambda_1}{1 - \epsilon'} I - A \right) \leq \tfrac{1}{3}\lambda_{\min}(A)$ in $O(\min\{n,\beta \log(n\delta^{-2})\})$ matrix-vector multiplications. Letting $\lambda_2 = \frac{2\lambda_1}{1 - \epsilon'} - \lambda'$, then we have $|\lambda_2-\lambda_{\min}(A)|\leq \frac{1}{3}\lambda_{\min}(A)$. Thus, letting $X^2=2\lambda_1/\lambda_2$ suffices. \Halmos
\end{proof}

\begin{lemma}{\cite[Theorem 5.5]{Wright-2006-Numerical}}
\label{lem:cg-apx}
Given a matrix $H \succ 0$ and $\beta \geq \sqrt{\mathrm{cond}(H)}$, CG-Inverse$(H, g, \beta, \alpha)$ returns  $h \in \mathbb{R}^n$ such that $\|h - H^{-1}g\|_H \leq \alpha \|g\|_{H^{-1}}$. The algorithm requires at most $C_{\mathrm{CG}}\min\{n,\, \beta \log(\alpha^{-1})\}$ matrix-vector multiplications with $H$, where $C_{\mathrm{CG}}>0$ is a constant. Moreover, we have $(1-\alpha)\|g\|_{H^{-1}} \leq \|h\|_H \leq (1+\alpha)\|g\|_{H^{-1}}$ and $(1-\alpha) g^\top H^{-1}g \leq h^\top g \leq (1+\alpha) g^\top H^{-1}g$.
\end{lemma}

\begin{lemma}
\label{lem:condition-numbers-of-two-points-near-optimal}
Let $f$ be $\kappa$-self-concordant with global minimizer $x_f^\star$. Let $x,y\in \dom(f)$, $R\in (0,0.5)$, and $\kappa>0$ such that $\lambda_f(x)\leq \frac{R}{\kappa}$ and $\lambda_f(y)\leq \frac{R}{\kappa}$. Then $\mathrm{cond}(\nabla^2f(x)) \leq \left(\tfrac{R^{-1}-1}{R^{-1}-2} \right)^4\mathrm{cond}(\nabla^2 f(x_f^\star))\leq \left(\tfrac{R^{-1}-1}{R^{-1}-2}\right)^8 \mathrm{cond}(\nabla^2f(y))$.
\end{lemma}
\begin{proof}{Proof}
We know from \cite[Theorem 5.1.13]{Nesterov-2018-Lectures} that $x_f^\star$ exists. It follows from \cite[Theorems 5.1.7 5.2.1]{Nesterov-2018-Lectures} that $(1-\kappa r)^2\nabla^2f(x) \preceq \nabla^2 f(x_{f}^\star) \preceq \frac{1}{(1-\kappa r)^2} \nabla^2 f(x)$, where $r=\|x-x_{f}^\star\|_{\nabla^2f(x)}\leq \frac{1}{(R^{-1}-1)\kappa}$. The desired result follows from direct calculation. \Halmos
\end{proof}

We are ready to analyze Algorithm~\ref{algorithm:Newton-CG}. Convergence of the algorithm consists of two stages. In the first stage, it converges to a neighborhood of the global minimizer (see Lemma~\ref{lem:newton-cg-one-step-bounds}). In the second stage, it converges $Q$-linearly to the global minimizer (see Lemma~\ref{lem:newton-cg-one-step-bounds-local}).

\begin{lemma}
\label{lem:newton-cg-one-step-bounds}
Suppose $f$ is $\kappa$-self-concordant. At any given iteration of Algorithm~\ref{algorithm:Newton-CG}, if $\beta_k^2 \geq \mathrm{cond}(\nabla^2 f(x_k))$ and $\rho_k > \frac{R_1^2}{\kappa^2}$, then \(f(x_k) - f(x_{k+1}) \geq \kappa^{-2} \, \omega\left(\tfrac{\sqrt{1 - \alpha_\star}}{1 + \alpha_\star} \, \kappa \sqrt{\rho_k} \right)\geq \kappa^{-2}\omega(R_3)\) and $\sqrt{\tfrac{\mathrm{cond}(\nabla^2 f(x_{k+1}))}{\mathrm{cond}(\nabla^2 f(x_k))}} \leq B_k\leq \exp\left(C_1 \kappa^2 (f(x_k) - f(x_{k+1}))\right)$.
\end{lemma}
\begin{proof}{Proof}
It follows from Lemma \ref{lem:cg-apx} that ${\frac{\rho_k}{\delta_k}}\in[\frac{1-\alpha_\star}{(1+\alpha_\star)^2},\frac{1+\alpha_\star}{(1-\alpha_\star)^2}]$. Thus, Proposition~\ref{prop:descent} implies that $f(x_k)-f(x_{k+1}) \geq \kappa^{-2} \omega\left(\tfrac{\kappa \rho_k}{\sqrt{\delta_k}}\right)\geq \kappa^{-2} \omega\left(\tfrac{\sqrt{1-\alpha_\star}}{1+\alpha_\star}\kappa\sqrt{\rho_k}\right)\geq \kappa^{-2}\omega(R_3)$. Let $r_k=\|x_{k+1}-x_k\|_{\nabla^2f(x_k)}=\frac{\rho_k \sqrt{\delta_k}}{\delta_k+\kappa\rho_k \sqrt{\delta_k}}$. Hence, we have $(1-\kappa r_k)^{-1}\leq 1+\frac{\sqrt{1+\alpha_\star}}{1-\alpha_\star}\kappa\sqrt{\rho_k}$. It follows from \cite[Theorem 5.1.7]{Nesterov-2018-Lectures} that $\sqrt{\tfrac{\mathrm{cond}(\nabla^2 f(x_{k+1}))}{\mathrm{cond}(\nabla^2 f(x_k))}} \leq (1-\kappa r_k)^{-2} \leq B_k$. Moreover, we have
$\tfrac{\log B_k}{\kappa^2 (f(x_{k})-f(x_{k+1}))} \leq \tfrac{2\log \left(1 + \frac{\sqrt{1 + \alpha_\star}}{1 - \alpha_\star} \, \kappa \sqrt{\rho_k} \right)}{\omega\left(\tfrac{\sqrt{1-\alpha_\star}}{1+\alpha_\star}\kappa\sqrt{\rho_k}\right)}\leq\tfrac{2\log \left(1 + \frac{\sqrt{1 + \alpha_\star}}{1 - \alpha_\star} R_1\right)}{\omega\left(\tfrac{\sqrt{1-\alpha_\star}}{1+\alpha_\star}R_1\right)} \leq C_1$,
where the second inequality holds by the monotonicity of $\frac{\log\left(1+\frac{\sqrt{1+\alpha_\star}}{1-\alpha_\star}t\right)}{\omega\left(\frac{\sqrt{1-\alpha_\star}}{1+\alpha_\star}t\right)}$ for $t>R_1$. This completes the proof. \Halmos
\end{proof}

\begin{lemma}
\label{lem:newton-cg-one-step-bounds-local}
Suppose $f$ is $\kappa$-self-concordant. At any given iteration of Algorithm~\ref{algorithm:Newton-CG}, if $\beta_k^2 \geq \mathrm{cond}\allowbreak(\nabla^2 f(x_k))$ and $\rho_k \leq \frac{R_1^2}{\kappa^2}$, then $(1 + \alpha_\star)\lambda_f(x_{k+1})^2 \leq C_2 \rho_k$.
\end{lemma}
\begin{proof}{Proof}
If $\rho_k \leq \frac{R_1^2}{\kappa^2}$, then we have $t_k = \frac{\rho_k}{\delta_k+2\kappa\rho_k\sqrt{\delta_k}}$. Letting $\lambda_k = \lambda_f(x_k)$, $H_k=\nabla^2f(x_k)$, $G_k = \int_0^1 \nabla^2f(x_k+ \tau (x_{k+1}-x_k)) d\tau$, $\zeta_k  = \kappa\sqrt{\delta_k}$, we have $\nabla f(x_{k+1}) = G_k (x_{k+1}-x_k) +\nabla f(x_k) = t_kG_kh_k + g_k$. It follows from \cite[Corollary 5.1.5]{Nesterov-2018-Lectures} that $(1-t_k\zeta_k+\frac{1}{3}t_k^2\zeta_k^2) H_k\preceq G_k \preceq \frac{1}{1-t_k\zeta_k} H_k$. Moreover, \cite[Lemma 5.1.7]{Nesterov-2018-Lectures} implies that $H_{k+1}^{-1} \preceq \frac{1}{(1-t_k\zeta_k)^2} H_k^{-1}$. Thus, we have $\lambda_f(x_{k+1})^2\leq\tfrac{1}{(1-t_k\zeta_k)^2} \|\nabla f(x_{k+1})\|_{H_{k}^{-1}}^2$, where
$$\|\nabla f(x_{k+1})\|_{H_{k}^{-1}}^2=\lambda_k^2 - 2t_k h_k^\top G_kh_k  + t_k^2 h_k^\top G_k H_{k}^{-1} G_k h_k + 2t_k h_k^\top G_k (H_{k}^{-1}g_k+h_k).$$ Using the Cauchy–Schwarz inequality and Lemma~\ref{lem:cg-apx}, we have
$$|2h_k^\top G_k (H_{k}^{-1}g_k+h_k)|\leq \alpha_\star \|G_k h_k\|_{H_k^{-1}}^2 + \alpha_\star^{-1}\|H_k^{-1}g_k+h_k\|_{H_k}^2.$$
Thus, we have $\lambda_{k+1}^2\leq \lambda_k^2g(t_{k},\zeta_k)$, where $$g(t,\zeta)=\tfrac{1-2t(1-\alpha_\star)^2(1-t\zeta+\frac{1}{3}t^2\zeta^2)}{(1-t\zeta)^2} + \tfrac{t^2 (1+\alpha_\star)^2}{(1-t\zeta)^4} + \tfrac{\alpha_\star t}{{(1-t\zeta)^2}}\left(\tfrac{(1+\alpha_\star)^2}{(1-t\zeta)^2} + 1\right).$$

Lemma~\ref{lem:cg-apx} implies $\zeta_k\in[0,\frac{1+\alpha_\star}{\sqrt{1-\alpha_\star}}R_1]$. For any $\zeta_l<\zeta_u$, we define $t_u = \frac{1}{\frac{(1-\alpha_\star)^2}{1+\alpha_\star}+3\zeta_l}$ and $t_l = \frac{1}{\frac{(1+\alpha_\star)^2}{1-\alpha_\star}+3\zeta_u}$. Since $t_k=\frac{1}{\frac{\delta_k}{\rho_k}+3\zeta_k}$ and $\frac{\delta_k}{\rho_k} \in [\frac{(1-\alpha_\star)^2}{1+\alpha_\star},\frac{(1+\alpha_\star)^2}{1-\alpha_\star}]$, we have
$g(t_k,\zeta_k) \leq \tfrac{1-2t_l(1-\alpha_\star)^2(1-t_u\zeta_u)}{(1-t_l\zeta_l)^2} + \tfrac{t_u^2 (1+\alpha_\star)^2}{(1-t_u\zeta_u)^4} + \tfrac{\alpha_\star t_u}{{(1-t_u\zeta_u)^2}}\left(\tfrac{(1+\alpha_\star)^2}{(1-t_u\zeta_u)^2} + 1\right)$
for any $\zeta_k\in[\zeta_l,\zeta_u]$. Dividing $[0,\frac{1+\alpha_\star}{\sqrt{1-\alpha_\star}}R_1]$ into $1000$ intervals equally, we can verify numerically that $g(t_k,\zeta_k)\leq0.945$. The desired result follows from Lemma~\ref{lem:cg-apx}. \Halmos
\end{proof}

In Algorithm~\ref{algorithm:Newton-CG}, we define $k^\star = \min\{k:\rho_k\leq R_1^2\kappa^{-2}\}$ and the events
\begin{align*}
A_1 &= \left\{ \mathrm{cond}(\nabla^2f(x_0))\leq \beta_0^2\leq  4\mathrm{cond}(\nabla^2f(x_0)) \right\},\\
A_2 &= \left\{\mathrm{cond}(\nabla^2f(x_{k^\star}))\leq (\beta^\star)^2 \leq  4\mathrm{cond}(\nabla^2f(x_{k^\star})) \right\}, \\
A_3 &= \left\{ \beta_k^2 \geq \mathrm{cond}(\nabla^2f(x_k)), \ \forall k\geq 0 \right\}.
\end{align*}
Lemma~\ref{lem:all-condition-numbers-are-bounded-by-beta} shows that $A_1\cap A_2\cap A_3$ happens with high probability. Hence, using two calls of sqrt-cond, Algorithm~\ref{algorithm:Newton-CG} gets good estimates of Hessian condition numbers. If $A_1\cap A_2\cap A_3$ holds, then for all $0\leq k\leq k^\star$, using Lemma~\ref{lem:newton-cg-one-step-bounds}, we have
\begin{equation}
\label{eq:bound-beta-k-global}
    \beta_k=\beta_0 B_0B_1\cdots B_{k-1}\leq 2\sqrt{\mathrm{cond}(\nabla^2f(x_0))}\exp(C_1\kappa^2(f(x_0)-f(x_{k}))).
\end{equation}
Moreover, for all $k\geq k^\star+1$, by Lemma~\ref{lem:condition-numbers-of-two-points-near-optimal}, we have
\begin{equation}
\label{eq:bound-beta-k-local}
    \beta_k =C_3^4 \beta^\star\leq 2C_3^4 \sqrt{\mathrm{cond}(\nabla^2f(x_{k^\star}))}\leq 2C_3^6\sqrt{\mathrm{cond}(\nabla^2f(\arg\min f))}.
\end{equation}

\begin{lemma}
\label{lem:all-condition-numbers-are-bounded-by-beta}
Suppose that $f$ is $\kappa$-self-concordant and $\beta^2 \geq \mathrm{cond}(\nabla^2f(x_0))$, then $\mathbb{P}\left( A_1 \cap A_2 \cap A_3 \right) \geq 1 - \delta$.
\end{lemma}
\begin{proof}{Proof}
By Lemma \ref{lem:lanczos}, we have $\mathbb{P}\left(A_1\right)\geq 1-\frac{1}{2}\delta$. Therefore, it suffices to show $\mathbb{P}\left(A_2|A_1\right)\geq 1-\frac{1}{2}\delta$ and $A_1\cap A_2 \subseteq A_3$. To show $\mathbb{P}\left(A_2|A_1\right)\geq 1-\frac{1}{2}\delta$, we observe from Lemma~\ref{lem:newton-cg-one-step-bounds} that $\beta_0^2\geq \mathrm{cond}(\nabla^2f(x_0))$ implies $\beta_{k^\star}^2\geq \mathrm{cond}(\nabla^2f(x_{k^\star}))$. Thus, Lemma \ref{lem:lanczos} implies $\mathbb{P}\left(A_2|A_1\right)\geq 1-\frac{1}{2}\delta$. It remains to show $A_1\cap A_2 \subseteq A_3$. Suppose that $A_1\cap A_2$ holds, then we have $\mathrm{cond}(\nabla^2 f(x_{l}))\leq\beta_l^2$ for $0\leq l\leq k^\star$ by Lemma~\ref{lem:newton-cg-one-step-bounds}. Moreover, Lemma~\ref{lem:newton-cg-one-step-bounds-local} yields $\rho_l \leq \frac{R_1^2}{\kappa^2}$ for all $l\geq k^\star+1$. Therefore, using Lemma~\ref{lem:condition-numbers-of-two-points-near-optimal}, we have $\mathrm{cond}(\nabla^2 f(x_{l}))\leq C_3^8\mathrm{cond}(\nabla^2 f(x_{k^\star})) \leq C_3^8(\beta^\star)^2 = \beta_{\mathrm{local}}^2 = \beta_l^2$.  \Halmos
\end{proof}

We present the convergence theorem for Algorithm~\ref{algorithm:Newton-CG}.

\begin{theorem}
\label{thm:convex-SC-newton-CG}
Suppose that $f$ is $\kappa$-self-concordant, bounded below, and that $\beta^2 \geq \mathrm{cond}(\nabla^2 f(x_0))$. Then, with probability at least $1 - \delta$, the following statements hold:
1. Algorithm~\ref{algorithm:Newton-CG} terminates and returns a point $y$. If $y \neq x_0$, then $\lambda_f(y) \leq \epsilon_1$ and $\lambda_f(x_{0}) \geq \sqrt{\frac{1-\alpha_\star}{1+\alpha_\star}} \,\epsilon_0$; if $y = x_0$, then $\lambda_f(x_{0}) \leq \epsilon_0$. Let $\Delta_f = \kappa^2(f(x_0) - f(y))$, $K_1 = \frac{\Delta_f}{\omega(R_3)}$, and $K_2 = 2 \log_{C_2^{-1}}\!\left(\frac{R_2}{\kappa \epsilon_1}\right)$. Then Algorithm~\ref{algorithm:Newton-CG} terminates within $2 + \lfloor K_1 \rfloor + \max(\lfloor K_2 \rfloor, 0)$ iterations.
2. Let $B = \sqrt{\max\left\{{\mathrm{cond}(\nabla^2 f(x_0))}, \mathrm{cond}(\nabla^2 f(\argmin f))\right\}}$, $\Lambda_1 = C_{\mathrm{CG}}\min\left\{n,2Be^{C_1 \Delta_f} \log(\alpha_\star^{-1}) \right\}$, $\Lambda_2 = C_{\mathrm{CG}}\min \left\{ n, 2 B C_3^6 \log(\alpha_\star^{-1})\right\}$, and $\Lambda_3 = C_L \min\left\{ n, \beta \log(n\delta^{-2}) \right\} + C_L\min\left\{ n, 2B e^{C_1 \Delta_f} \log(n\delta^{-2}) \right\}$. The total number of HVP computations performed by Algorithm~\ref{algorithm:Newton-CG} is at most $(\Lambda_1+1)(K_1+1) + (\Lambda_2+1)(K_2+1) + \Lambda_3 = O(\log(\epsilon_1^{-1}) + \log(\delta^{-1}))$.
\end{theorem}

\begin{proof}{Proof}
By Lemma~\ref{lem:all-condition-numbers-are-bounded-by-beta}, it suffices to prove all statements hold deterministically assuming that the event $A_1 \cap A_2 \cap A_3$ holds. For all $k = 0, \ldots, k^\star - 1$, we have $\rho_k > \frac{R_1^2}{\kappa^2}$. Therefore, it follows from Lemma~\ref{lem:newton-cg-one-step-bounds} that $f(x_k)-f(x_{k+1}) \geq \kappa^{-2}\omega(R_3)$. Summing this inequality over $k=0,1,\dots,k^\star-1$ yields ${k^\star}\kappa^{-2}\omega(R_3) \leq f(x_0)-f(x_{k^\star})$. Thus, $k^\star$ must be finite. Lemma~\ref{lem:newton-cg-one-step-bounds-local} implies that we have $\rho_{k+1}\leq C_2\rho_k\leq \frac{R_1^2}{\kappa^2}$ for all $k\geq k^\star$. Thus, we have $(1-\alpha_\star)\epsilon_1^2 \leq \rho_k \leq C_2^{k-k^\star} \rho_{k^\star}\leq C_2^{k-k^\star} \frac{R_1^2}{\kappa^2}$. Therefore, the algorithm terminates at some $k\leq1+\lfloor K_1\rfloor+\max(\lfloor K_2\rfloor,0)$.

Suppose that Algorithm~\ref{algorithm:Newton-CG} returns $x_{M}\not=x_0$, then we have $\rho_0\geq (1-\alpha_\star)\epsilon_0^2$. It follows from Lemma~\ref{lem:cg-apx} that $\lambda_f(x_0)\geq \sqrt{\frac{1-\alpha_\star}{1+\alpha_\star}} \epsilon_0$. Moreover, since Algorithm~\ref{algorithm:Newton-CG} terminates at $x_M$, we have $\rho_M\leq (1-\alpha_\star)\epsilon^2_1$, and thus $\lambda_f(x_M)\leq \epsilon_1$. Suppose that Algorithm~\ref{algorithm:Newton-CG} returns $x_0$, then we have $\lambda_f(x_0)\leq \epsilon_0$ by Lemma~\ref{lem:cg-apx}.

Now, we can estimate the total number of HVPs. By Lemma~\ref{lem:cg-apx}, the total number of HVPs called by CG-Inverse is bounded by $\sum_{k=0}^{k^\star} C_{\mathrm{CG}}\min\{n, \beta_k\log(\alpha_\star^{-1})\} + \sum_{k=1}^{\lfloor K_2 \rfloor+1} C_{\mathrm{CG}}\min\{n, \beta_{k+k^\star}\log(\alpha_\star^{-1})\}$. Moreover, we call sqrt-cond at the zeroth iteration and the $k^\star$-th iteration. Hence, the total number of HVP computations called by sqrt-cond is at most $\Lambda_3$, by Lemma \ref{lem:lanczos} and Eq.~\eqref{eq:bound-beta-k-local}. One HVP is computed in each iteration additionally. Putting these pieces together, the desired result follows from Eq.~\eqref{eq:bound-beta-k-global} and Eq.~\eqref{eq:bound-beta-k-local}. \Halmos
\end{proof}

We are ready to prove the convergence of IPPM, which shares a similar spirit with \cite{Carmon-2018-Accelerated}.
\begin{proof}{Proof of Theorem~\ref{thm:IPPM}}
With probability at least $1-(K+2)\delta'\geq 1-\delta$, the results of Theorem~\ref{thm:convex-SC-newton-CG} hold for the first $\lfloor K \rfloor+2$ calls of Algorithm~\ref{algorithm:Newton-CG}. Hence, if we are able to prove that IPPM must terminate in $\lfloor K \rfloor+2$ iterations under the assumption that the results of Theorem~\ref{thm:convex-SC-newton-CG} hold for all calls of Algorithm~\ref{algorithm:Newton-CG}, then indeed IPPM terminates in $\lfloor K \rfloor+2$ iterations with probability at least $1-\delta$.

It follows from Proposition~\ref{prop:descent}, Lemma~\ref{lem:newton-cg-one-step-bounds}, and Lemma~\ref{lem:newton-cg-one-step-bounds-local} that $f(z_{j+1}) + \tfrac{1}{2}\mu\|z_{j+1}-z_j\|^2= f_j(z_{j+1}) \leq f_j(z_j) = f(z_j)$. If we have $z_{j}\not=z_{j+1}$ for $j=0,1,\dots,N-1$ , then summing this inequality over $j=0,1,\dots,N-1$ yields
\begin{equation}
\label{eq:ippm-outer}
    \tfrac{1}{2}\mu\sum_{j=0}^{N-1} \|z_{j+1}-z_j\|^2 \leq f(z_0) - f(z_N) \leq f(z_0) - \inf_y f(y).
\end{equation}
Let $H_{j} = \nabla^2f(z_{j})+\mu I$. By the fact that $\lambda_{\max}(H_{j}^{-1}) = (\lambda_{\min}(H_{j}))^{-1} \leq \frac{1}{\mu-\ell}$ and the definition of $f_j$, we have $\|z_{j+1}-z_j\|^2\geq (\mu-\ell)(z_{j+1}-z_j)^{\top}H_{j+1}^{-1}(z_{j+1}-z_j) = \frac{\mu-\ell}{\mu^2}\|\nabla f(z_{j+1}) - \nabla f_j(z_{j+1})\|^2_{H_{j+1}^{-1}}$. Thus, we have $\|z_{j+1}-z_j\|^2 \geq \frac{\mu-\ell}{\mu^2}(\|\nabla f(z_{j+1})\|_{H_{j+1}^{-1}} - \|\nabla f_j(z_{j+1})\|_{H_{j+1}^{-1}})^2
= \frac{\mu-\ell}{\mu^2}(\lambda_{f_{j+1}}(z_{j+1}) - \lambda_{f_j}(z_{j+1}))^2
\geq \frac{\mu-\ell}{\mu^2} (\sqrt{\frac{1-\alpha_\star}{1+\alpha_\star}}\epsilon-\epsilon')^2 = \frac{\mu-\ell}{4\mu^2}\epsilon^2$, where the last inequality follows from Theorem \ref{thm:convex-SC-newton-CG}. Therefore, we have $\|z_{j+1}-z_j\|^2\geq \frac{\mu-\ell}{4\mu^2}\epsilon^2$ and thus $N \leq \tfrac{8\mu(f(z_0)-\inf f)}{(\mu-\ell)\epsilon^2}$ by Eq.~\eqref{eq:ippm-outer}.

Suppose that IPPM terminates at $z_{N}=z_{N+1}$, then we have $N\leq \tfrac{8\mu(f(z_0)-\inf f)}{(\mu-\ell)\epsilon^2}$. This implies that IPPM terminates in $\lfloor K\rfloor + 2$ iterations. It follows from Theorem \ref{thm:convex-SC-newton-CG} that $\lambda_{f_{N+1}}(z_{N+1})=\nu_{f,\frac{1}{2}\mu\|\cdot\|^2}(z_{N+1}) \leq \epsilon$.

Let $\Lambda_1 = C_{\mathrm{CG}}\min \left\{ n,\, 2 B e^{C_1 \Delta_f} \log(\alpha_\star^{-1}) \right\}$, $\Lambda_2 = C_{\mathrm{CG}}\min \left\{ n,\, 2 B C_3^6 \log(\alpha_\star^{-1})\right\}$, and $\Lambda_3 = C_L (\min\{ n,\allowbreak\, \beta' \log(n\delta'^{-2}) \} + \min\{ n,\, 2B e^{C_1 \Delta_f} \log(n\delta'^{-2})\})$. For $j\leq N$, we define $\Delta_f^{(j)} = f_{j}(z_{j})-f_j(z_{j+1})$ and $K_1^{(j)} = \frac{\Delta_f^{(j)}}{\omega(R_3)}$. Moreover, we define $\Delta_f^{(N+1)} = 0$ and $K_2 = 2\log_{C_2^{-1}} \frac{R_2}{\kappa \epsilon'}$.  Theorem~\ref{thm:convex-SC-newton-CG} yields that the total number of HVP computations by IPPM is bounded by $\sum_{j=0}^{N+1} [(1+K_1^{(j)})(1+\Lambda_1) + (1+K_2)(1+\Lambda_2) + \Lambda_{3}]\leq \tfrac{1+\Lambda_1}{\omega(R_3)} \sum_{j=0}^{N+1} \Delta_f^{(j)} + (1+\Lambda_2) (N+2)2\log_{C_2^{-1}} \frac{R_2}{\kappa \epsilon'}+ (N+2)(2+\Lambda_1+\Lambda_2+ \Lambda_3)$. The desired result follows from $\sum_{j=0}^{N+1} \Delta_f^{(j)} = \sum_{j=0}^{N+1} (f_{j}(z_{j}) - f_{j}(z_{j+1})) \allowbreak=\sum_{j=0}^{N+1} (f(z_{j}) - f(z_{j+1}) - \frac{1}{2}\mu\|z_{j+1}-z_j\|^2) \leq f(z_0)-\inf f$. \Halmos
\end{proof}

\bibliographystyle{abbrv}
\bibliography{ref}

@Incollection{Boerner-2023-Access,
    Title         = {Access: Advancing innovation: {NSF}’s advanced cyberinfrastructure coordination ecosystem: Services \& support},
    Author        = {Boerner, T. J. and Deems, S. and Furlani, T. R. and Knuth, S. L. and Towns, J. },
    Booktitle     = {Practice and Experience in Advanced Research Computing 2023: Computing for the Common Good},
    Pages         = {173-176},
    Year          = {2023}, 
    Publisher     = {Association for Computing Machinery (ACM)}
}

@Book{Wright-2006-Numerical,
    Title         = {Numerical Optimization},
    Author        = {Nocedal, J. and Wright, S. J. },
    Year          = {2006},
    Publisher     = {Springer}
}

@Article{Cartis-2010-Complexity,
    Title         = {On the complexity of steepest descent, {N}ewton's and regularized {N}ewton's methods for nonconvex unconstrained optimization problems},
    Author        = {Cartis, C. and Gould, N. I. M. and Toint, P. L. },
    Journal       = {SIAM Journal on Optimization},
    Volume        = {20},
    Number        = {6},
    Pages         = {2833-2852},
    Year          = {2010},
    Publisher     = {SIAM}
}

@Article{Nesterov-2006-Cubic,
    Title         = {Cubic regularization of {N}ewton method and its global performance},
    Author        = {Nesterov, Y. and Polyak, B. T. },
    Journal       = {Mathematical Programming},
    Volume        = {108},
    Number        = {1},
    Pages         = {177-205},
    Year          = {2006},
    Publisher     = {Springer}
}

@Article{Nesterov-2008-Accelerating,
    Title         = {Accelerating the cubic regularization of {N}ewton’s method on convex problems},
    Author        = {Nesterov, Y. },
    Journal       = {Mathematical Programming},
    Volume        = {112},
    Number        = {1},
    Pages         = {159-181},
    Year          = {2008},
    Publisher     = {Springer}
}

@Article{Cartis-2011-Adaptive,
    Title         = {Adaptive cubic regularisation methods for unconstrained optimization. {P}art {I}: motivation, convergence and numerical results},
    Author        = {Cartis, C. and Gould, N. I. M. and Toint, P. L. },
    Journal       = {Mathematical Programming},
    Volume        = {127},
    Number        = {2},
    Pages         = {245-295},
    Year          = {2011},
    Publisher     = {Springer}
}

@Article{Cartis-2011-Adaptive2,
    Title         = {Adaptive cubic regularisation methods for unconstrained optimization. {P}art {II}: worst-case function-and derivative-evaluation complexity},
    Author        = {Cartis, C. and Gould, N. I. M. and Toint, P. L. },
    Journal       = {Mathematical Programming},
    Volume        = {130},
    Number        = {2},
    Pages         = {295-319},
    Year          = {2011},
    Publisher     = {Springer}
}

@Book{Conn-2000-Trust,
    Title         = {Trust Region Methods},
    Author        = {Conn, A. R. and Gould, N. I. M. and Toint, P. L. },
    Year          = {2000},
    Publisher     = {SIAM}
}

@Article{Curtis-2017-Trust,
    Title         = {A trust region algorithm with a worst-case iteration complexity of $\mathcal{O}(\epsilon^{-3/2})$ for nonconvex optimization},
    Author        = {Curtis, F. E. and Robinson, D. P. and Samadi, M. },
    Journal       = {Mathematical Programming},
    Volume        = {162},
    Pages         = {1-32},
    Year          = {2017},
    Publisher     = {Springer}
}

@Article{Curtis-2018-Concise,
    Title         = {Concise complexity analyses for trust region methods},
    Author        = {Curtis, F. E. and Lubberts, Z. and Robinson, D. P. },
    Journal       = {Optimization Letters},
    Volume        = {12},
    Number        = {8},
    Pages         = {1713-1724},
    Year          = {2018},
    Publisher     = {Springer}
}

@Article{Curtis-2023-Worst,
    Title         = {Worst-case complexity of trace with inexact subproblem solutions for nonconvex smooth optimization},
    Author        = {Curtis, F. E. and Wang, Q. },
    Journal       = {SIAM Journal on Optimization},
    Volume        = {33},
    Number        = {3},
    Pages         = {2191-2221},
    Year          = {2023},
    Publisher     = {SIAM}
}

@Article{Mishchenko-2023-Regularized,
    Title         = {Regularized {N}ewton method with global {{\({\boldsymbol{\mathcal O(1/{k}^2)}}\)}} convergence},
    Author        = {Mishchenko, K. },
    Journal       = {SIAM Journal on Optimization},
    Volume        = {33},
    Number        = {3},
    Pages         = {1440-1462},
    Year          = {2023}, 
    Publisher     = {SIAM}
}

@Article{Doikov-2024-Gradient,
    Title         = {Gradient regularization of {N}ewton method with {B}regman distances},
    Author        = {Doikov, N. and Nesterov, Y. },
    Journal       = {Mathematical Programming},
    Volume        = {204},
    Number        = {1},
    Pages         = {1-25},
    Year          = {2024},
    Publisher     = {Springer}
}

@Article{Gratton-2025-Yet,
    Title         = {Yet another fast variant of {N}ewton’s method for nonconvex optimization},
    Author        = {Gratton, S. and Jerad, S. and Toint, P. L. },
    Journal       = {IMA Journal of Numerical Analysis},
    Volume        = {45},
    Number        = {2},
    Pages         = {971-1008},
    Year          = {2025},
    Publisher     = {Oxford University Press}
}

@Book{Nesterov-1994-Interior,
    Title         = {Interior-Point Polynomial Algorithms in Convex Programming},
    Author        = {Nesterov, Y. and Nemirovskii, A. },
    Year          = {1994},
    Publisher     = {SIAM}
}

@Book{Nesterov-2018-Lectures,
    Title         = {Lectures on Convex Optimization},
    Author        = {Nesterov, Y. },
    Volume        = {137},
    Year          = {2018},
    Publisher     = {Springer}
}

@Article{Tran-2015-Composite,
    Title         = {Composite self-concordant minimization},
    Author        = {Tran-Dinh, Q. and Kyrillidis, A. and Cevher, V. },
    Journal       = {The Journal of Machine Learning Research},
    Volume        = {16},
    Number        = {1},
    Pages         = {371-416},
    Year          = {2015}, 
    Publisher     = {JMLR. org}
}

@Inproceedings{Zhang-2015-Disco,
    Title         = {{DISCO}: Distributed optimization for self-concordant empirical loss},
    Author        = {Zhang, Y. and Lin, X. },
    Booktitle     = {ICML},
    Pages         = {362-370},
    Year          = {2015},
    Organization  = {PMLR}
}

@Article{Lu-2017-Randomized,
    Title         = {Randomized block proximal damped {N}ewton method for composite self-concordant minimization},
    Author        = {Lu, Z. },
    Journal       = {SIAM Journal on Optimization},
    Volume        = {27},
    Number        = {3},
    Pages         = {1910-1942},
    Year          = {2017},
    Publisher     = {SIAM}
}

@Article{Gao-2019-Quasi,
    Title         = {Quasi-{N}ewton methods: Superlinear convergence without line searches for self-concordant functions},
    Author        = {Gao, W. and Goldfarb, D. },
    Journal       = {Optimization Methods and Software},
    Volume        = {34},
    Number        = {1},
    Pages         = {194-217},
    Year          = {2019},
    Publisher     = {Taylor \& Francis}
}

@Article{Rodomanov-2021-Greedy,
    Title         = {Greedy quasi-{N}ewton methods with explicit superlinear convergence},
    Author        = {Rodomanov, A. and Nesterov, Y. },
    Journal       = {SIAM Journal on Optimization},
    Volume        = {31},
    Number        = {1},
    Pages         = {785-811},
    Year          = {2021},
    Publisher     = {SIAM}
}

@Inproceedings{Hanzely-2022-Damped,
    Title         = {A damped {N}ewton method achieves global $O(1/k^2)$ and local quadratic convergence rate},
    Author        = {Hanzely, S. and Kamzolov, D. and Pasechnyuk, D. and Gasnikov, A. and Richt{\'a}rik, P. and Tak{\'a}{\v{c}}, M. },
    Booktitle     = {NeurIPS},
    Pages         = {25320-25334},
    Year          = {2022}
}

@Article{Bach-2010-Self,
    Title         = {Self-concordant analysis for logistic regression},
    Author        = {Bach, F. },
    Journal       = {Electronic Journal of Statistics},
    Volume        = {4},
    Pages         = {384-414},
    Year          = {2010}
}

@Article{Bach-2014-Adaptivity,
    Title         = {Adaptivity of averaged stochastic gradient descent to local strong convexity for logistic regression},
    Author        = {Bach, F. },
    Journal       = {The Journal of Machine Learning Research},
    Volume        = {15},
    Number        = {1},
    Pages         = {595-627},
    Year          = {2014},
    Publisher     = {JMLR. org}
}

@Inproceedings{Dvurechensky-2020-Self,
    Title         = {Self-concordant analysis of {F}rank-{W}olfe algorithms},
    Author        = {Dvurechensky, P. and Ostroukhov, P. and Safin, K. and Shtern, S. and Staudigl, M. },
    Booktitle     = {ICML},
    Pages         = {2814-2824},
    Year          = {2020},
    Organization  = {PMLR}
}

@Inproceedings{Carderera-2021-Simple,
    Title         = {Simple steps are all you need: {F}rank-{W}olfe and generalized self-concordant functions},
    Author        = {Carderera, A. and Besan{\c{c}}on, M. and Pokutta, S. },
    Booktitle     = {NeurIPS},
    Pages         = {5390-5401},
    Year          = {2021}
}

@Article{Dvurechensky-2023-Generalized,
    Title         = {Generalized self-concordant analysis of {F}rank-{W}olfe algorithms},
    Author        = {Dvurechensky, P. and Safin, K. and Shtern, S. and Staudigl, M. },
    Journal       = {Mathematical Programming},
    Volume        = {198},
    Number        = {1},
    Pages         = {255-323},
    Year          = {2023},
    Publisher     = {Springer}
}

@Article{Curtis-2021-Trust,
    Title         = {Trust-region {N}ewton-{CG} with strong second-order complexity guarantees for nonconvex optimization},
    Author        = {Curtis, F. E. and Robinson, D. P and Royer, C. W. and Wright, S. J. },
    Journal       = {SIAM Journal on Optimization},
    Volume        = {31},
    Number        = {1},
    Pages         = {518-544},
    Year          = {2021},
    Publisher     = {SIAM}
}

@Article{Dussault-2024-Scalable,
    Title         = {Scalable adaptive cubic regularization methods},
    Author        = {Dussault, J-P. and Migot, T. and Orban, D. },
    Journal       = {Mathematical Programming},
    Volume        = {207},
    Number        = {1},
    Pages         = {191-225},
    Year          = {2024},
    Publisher     = {Springer}
}

@Article{Carmon-2018-Accelerated,
    Title         = {Accelerated methods for nonconvex optimization},
    Author        = {Carmon, Y. and Duchi, J. C. and Hinder, O. and Sidford, A. },
    Journal       = {SIAM Journal on Optimization},
    Volume        = {28},
    Number        = {2},
    Pages         = {1751-1772},
    Year          = {2018},
    Publisher     = {SIAM}
}

@Inproceedings{Agarwal-2017-Finding,
    Title         = {Finding approximate local minima faster than gradient descent},
    Author        = {Agarwal, N. and Allen-Zhu, Z. and Bullins, B. and Hazan, E. and Ma, T. },
    Booktitle     = {STOC},
    Pages         = {1195-1199},
    Year          = {2017}
}

@Article{Mccormick-1977-Modification,
    Title         = {A modification of {A}rmijo's step-size rule for negative curvature},
    Author        = {McCormick, G. P. },
    Journal       = {Mathematical Programming},
    Volume        = {13},
    Number        = {1},
    Pages         = {111-115},
    Year          = {1977},
    Publisher     = {Springer}
}

@Article{More-1979-Use,
    Title         = {On the use of directions of negative curvature in a modified {N}ewton method},
    Author        = {Mor{\'e}, J. J. and Sorensen, D. C. },
    Journal       = {Mathematical Programming},
    Volume        = {16},
    Number        = {1},
    Pages         = {1-20},
    Year          = {1979},
    Publisher     = {Springer}
}

@Techreport{Goldfarb-1980-Use,
    Title         = {The use of negative curvature in minimization algorithms},
    Author        = {Goldfarb, D. },
    Year          = {1980},
    Institution   = {Cornell University}
}

@Article{Royer-2018-Complexity,
    Title         = {Complexity analysis of second-order line-search algorithms for smooth nonconvex optimization},
    Author        = {Royer, C. W. and Wright, S. J. },
    Journal       = {SIAM Journal on Optimization},
    Volume        = {28},
    Number        = {2},
    Pages         = {1448-1477},
    Year          = {2018},
    Publisher     = {SIAM}
}

@Article{Royer-2020-Newton,
    Title         = {A {N}ewton-{CG} algorithm with complexity guarantees for smooth unconstrained optimization},
    Author        = {Royer, C. W. and O'Neill, M. and Wright, S. J. },
    Journal       = {Mathematical Programming},
    Volume        = {180},
    Number        = {1},
    Pages         = {451-488},
    Year          = {2020},
    Publisher     = {Springer}
}

@Article{Kuczynski-1992-Estimating,
    Title         = {Estimating the largest eigenvalue by the power and {L}anczos algorithms with a random start},
    Author        = {Kuczy{\'n}ski, J. and Wo{\'z}niakowski, H. },
    Journal       = {SIAM Journal on Matrix Analysis and Applications},
    Volume        = {13},
    Number        = {4},
    Pages         = {1094-1122},
    Year          = {1992},
    Publisher     = {SIAM}
}

@Article{Davis-2019-Stochastic,
    Title         = {Stochastic model-based minimization of weakly convex functions},
    Author        = {Davis, D. and Drusvyatskiy, D. },
    Journal       = {SIAM Journal on Optimization},
    Volume        = {29},
    Number        = {1},
    Pages         = {207-239},
    Year          = {2019},
    Publisher     = {SIAM}
}

@Article{Sun-2018-Geometric,
    Title         = {A geometric analysis of phase retrieval},
    Author        = {Sun, J. and Qu, Q. and Wright, J. },
    Journal       = {Foundations of Computational Mathematics},
    Volume        = {18},
    Number        = {5},
    Pages         = {1131-1198},
    Year          = {2018},
    Publisher     = {Springer}
}

@Article{Tang-2024-Riemannian,
    Title         = {A {R}iemannian dimension-reduced second-order method with application in sensor network localization},
    Author        = {Tang, T. and Toh, K-C. and Xiao, N. and Ye, Y. },
    Journal       = {SIAM Journal on Scientific Computing},
    Volume        = {46},
    Number        = {3},
    Pages         = {A2025-A2046},
    Year          = {2024},
    Publisher     = {SIAM}
}

@Inproceedings{Schmidt-2007-Fast,
    Title         = {Fast optimization methods for l1 regularization: A comparative study and two new approaches},
    Author        = {Schmidt, Mark and Fung, Glenn and Rosales, R{\'o}mer},
    Booktitle     = {ECML},
    Pages         = {286-297},
    Year          = {2007},
    Organization  = {Springer}
}

@Inproceedings{Mairal-2008-Supervised,
    Title         = {Supervised dictionary learning},
    Author        = {Mairal, J. and Bach, F. and Ponce, J. and Sapiro, G. and Zisserman, A. },
    Booktitle     = {NeurIPS},
    Pages         = {1033-1040},
    Year          = {2008}
}

@Inproceedings{Lee-2000-Algorithms,
    Title         = {Algorithms for non-negative matrix factorization},
    Author        = {Lee, D. and Seung, S. },
    Booktitle     = {NeurIPS},
    Pages         = {535-541},
    Year          = {2000}
}

@Article{Lojasiewicz-1963-Propriete,
    Title         = {Une propri{\'e}t{\'e} topologique des sous-ensembles analytiques r{\'e}els},
    Author        = {\L{}ojasiewicz, S. },
    Journal       = {Les {\'e}quations aux d{\'e}riv{\'e}es partielles},
    Volume        = {117},
    Number        = {87-89},
    Year          = {1963}
}

@Article{Polyak-1964-Gradient,
    Title         = {Gradient methods for solving equations and inequalities},
    Author        = {Polyak, B. T. },
    Journal       = {USSR Computational Mathematics and Mathematical Physics},
    Volume        = {4},
    Number        = {6},
    Pages         = {17-32},
    Year          = {1964},
    Publisher     = {Elsevier}
}

@Inproceedings{Zhang-2024-Transformers,
    Title         = {Why transformers need {ADAM}: A {H}essian perspective},
    Author        = {Zhang, Y. and Chen, C. and Ding, T. and Li, Z. and Sun, R. and Luo, Z-Q. },
    Booktitle     = {NeurIPS},
    Pages         = {131786-131823},
    Year          = {2024}
}

@Book{Cartis-2022-Evaluation,
    Title         = {Evaluation Complexity of Algorithms for Nonconvex Optimization: Theory, Computation and Perspectives},
    Author        = {Cartis, C. and Gould, N. I. M. and Toint, P. L. },
    Year          = {2022},
    Publisher     = {SIAM}
}

@Article{Absil-2005-Convergence,
    Title         = {Convergence of the iterates of descent methods for analytic cost functions},
    Author        = {Absil, P-A. and Mahony, R. and Andrews, B. },
    Journal       = {SIAM Journal on Optimization},
    Volume        = {16},
    Number        = {2},
    Pages         = {531-547},
    Year          = {2005},
    Publisher     = {SIAM}
}

@Article{Virtanen-2020-Scipy,
    Title         = {Sci{P}y 1.0: Fundamental algorithms for scientific computing in {P}ython},
    Author        = {Virtanen, P. and Gommers, R. and Oliphant, T. E. and Haberland, M. and Reddy, T. and Cournapeau, D. and Burovski, E. and Peterson, P. and Weckesser, W. and Bright, J. and others},
    Journal       = {Nature Methods},
    Volume        = {17},
    Number        = {3},
    Pages         = {261-272},
    Year          = {2020},
    Publisher     = {Nature Publishing Group US New York}
}

@Inproceedings{Martens-2015-Optimizing,
    Title         = {Optimizing neural networks with {K}ronecker-factored approximate curvature},
    Author        = {Martens, J. and Grosse, R. },
    Booktitle     = {ICML},
    Pages         = {2408-2417},
    Year          = {2015},
    Organization  = {PMLR}
}

@Inproceedings{Zhang-2019-Three,
    Title         = {Three mechanisms of weight decay regularization},
    Author        = {Zhang, G. and Wang, C. and Xu, B. and Grosse, R. },
    Booktitle     = {ICLR},
    Year          = {2019},
}

@Incollection{Martens-2012-Training,
    Title         = {Training deep and recurrent networks with {H}essian-free optimization},
    Author        = {Martens, J. and Sutskever, I. },
    Booktitle     = {Neural Networks: Tricks of the Trade: Second Edition},
    Pages         = {479-535},
    Year          = {2012},
    Publisher     = {Springer}
}

@Inproceedings{Pascanu-2014-Revisiting,
    Title         = {Revisiting natural gradient for deep networks},
    Author        = {Pascanu, R. and Bengio, Y. },
    Booktitle     = {ICLR},
    Year          = {2014},
}

@Article{Martens-2020-New,
    Title         = {New insights and perspectives on the natural gradient method},
    Author        = {Martens, J. },
    Journal       = {The Journal of Machine Learning Research},
    Volume        = {21},
    Number        = {146},
    Pages         = {1-76},
    Year          = {2020}, 
    Publisher     = {JMLR. org}
}

@Article{Abreu-2025-Potential,
    Title         = {The potential of second-order optimization for {LLM}s: A study with full {G}auss-{N}ewton},
    Author        = {Abreu, N. and Vyas, N. and Kakade, S. and Morwani, D. },
    Journal       = {ArXiv Preprint: 2510.09378},
    Year          = {2025}
}

@Inproceedings{Simonyan-2015-Very,
    Title         = {Very deep convolutional networks for large-scale image recognition},
    Author        = {Simonyan, K. and Zisserman, A. },
    Booktitle     = {ICLR},
    Year          = {2015},
}

@Inproceedings{He-2016-Deep,
    Title         = {Deep residual learning for image recognition},
    Author        = {He, K. and Zhang, X. and Ren, S. and Sun, J. },
    Booktitle     = {CVPR},
    Pages         = {770-778},
    Year          = {2016}
}

@Inproceedings{Zagoruyko-2016-Wide,
    Title         = {Wide residual networks},
    Author        = {Zagoruyko, S. and Komodakis, N. },
    Booktitle     = {BMVC},
    Pages         = {87-1},
    Year          = {2016},
    Organization  = {British Machine Vision Association}
}

@article{Krizhevsky-2009-Learning,
  title={Learning multiple layers of features from tiny images},
  author={Krizhevsky, Alex and Hinton, Geoffrey and others},
  year={2009},
  publisher={Toronto, ON, Canada}
}

@Inproceedings{Kingma-2015-Adam,
    Title         = {Adam: A method for stochastic optimization},
    Author        = {Kingma, D. P. and Ba, J. },
    Booktitle     = {ICLR},
    Year          = {2015},
}

@Incollection{Strande-2021-Expanse,
    Title         = {Expanse: Computing without boundaries: Architecture, deployment, and early operations experiences of a supercomputer designed for the rapid evolution in science and engineering},
    Author        = {Strande, S. and Cai, H. and Tatineni, M. and Pfeiffer, W. and Irving, C. and Majumdar, A. and Mishin, D. and Sinkovits, R. and Norman, M. and Wolter, N. and others},
    Booktitle     = {Practice and Experience in Advanced Research Computing 2021: Evolution Across All Dimensions},
    Pages         = {1-4},
    Year          = {2021}, 
    Publisher     = {Association for Computing Machinery (ACM)}
}

@article{Moreau-1965-Proximite,
  title={Proximit{\'e} et dualit{\'e} dans un espace hilbertien},
  author={Moreau, Jean-Jacques},
  journal={Bulletin de la Soci{\'e}t{\'e} math{\'e}matique de France},
  volume={93},
  pages={273--299},
  year={1965}
}

@article{rockafellar1976monotone,
  title={Monotone operators and the proximal point algorithm},
  author={Rockafellar, R Tyrrell},
  journal={SIAM journal on Control and Optimization},
  volume={14},
  number={5},
  pages={877--898},
  year={1976},
  publisher={SIAM}
}

@article{Zhou-2025-Regularized,
  title={A Regularized Newton Method for Nonconvex Optimization with Global and Local Complexity Guarantees},
  author={Zhou, Yuhao and Xu, Jintao and Li, Bingrui and Bao, Chenglong and Ding, Chao and Zhu, Jun},
  journal={arXiv preprint arXiv:2502.04799},
  year={2025}
}

@article{Gould-1999-Solving,
  title={Solving the trust-region subproblem using the Lanczos method},
  author={Gould, Nicholas IM and Lucidi, Stefano and Roma, Massimo and Toint, Philippe L},
  journal={SIAM Journal on Optimization},
  volume={9},
  number={2},
  pages={504--525},
  year={1999},
  publisher={SIAM}
}

@article{Adachi-2017-Solving,
  title={Solving the trust-region subproblem by a generalized eigenvalue problem},
  author={Adachi, Satoru and Iwata, Satoru and Nakatsukasa, Yuji and Takeda, Akiko},
  journal={SIAM Journal on Optimization},
  volume={27},
  number={1},
  pages={269--291},
  year={2017},
  publisher={SIAM}
}

@article{Sun-2019-Generalized,
  title={Generalized self-concordant functions: a recipe for newton-type methods},
  author={Sun, Tianxiao and Tran-Dinh, Quoc},
  journal={Mathematical Programming},
  volume={178},
  number={1},
  pages={145--213},
  year={2019},
  publisher={Springer}
}

@article{Li-2001-Global,
  title={On the global convergence of the BFGS method for nonconvex unconstrained optimization problems},
  author={Li, Dong-Hui and Fukushima, Masao},
  journal={SIAM Journal on Optimization},
  volume={11},
  number={4},
  pages={1054--1064},
  year={2001},
  publisher={SIAM}
}

@article{Cartis-2012-Adaptive,
  title={An adaptive cubic regularization algorithm for nonconvex optimization with convex constraints and its function-evaluation complexity},
  author={Cartis, Coralia and Gould, Nicholas IM and Toint, Ph L},
  journal={IMA Journal of Numerical Analysis},
  volume={32},
  number={4},
  pages={1662--1695},
  year={2012},
  publisher={OUP}
}

@article{Wachter-2006-Implementation,
  title={On the implementation of an interior-point filter line-search algorithm for large-scale nonlinear programming},
  author={W{\"a}chter, Andreas and Biegler, Lorenz T},
  journal={Mathematical Programming},
  volume={106},
  number={1},
  pages={25--57},
  year={2006},
  publisher={Springer}
}

@article{Ye-1994-nl,
  title={An $\mathcal{O}(\sqrt{nL})$-iteration homogeneous and self-dual linear programming algorithm},
  author={Ye, Yinyu and Todd, Michael J and Mizuno, Shinji},
  journal={Mathematics of Operations Research},
  volume={19},
  number={1},
  pages={53--67},
  year={1994},
  publisher={INFORMS}
}

@article{Zhang-2025-Homogeneous,
  title={A Homogeneous Second-Order Descent Method for Nonconvex Optimization},
  author={Zhang, Chuwen and He, Chang and Jiang, Yuntian and Xue, Chenyu and Jiang, Bo and Ge, Dongdong and Ye, Yinyu},
  journal={Mathematics of Operations Research},
  year={2025},
  publisher={INFORMS},
}

@article{He-2025-Newton,
  title={Newton-CG methods for nonconvex unconstrained optimization with H{\"o}lder continuous Hessian},
  author={He, Chuan and Huang, Heng and Lu, Zhaosong},
  journal={Mathematics of Operations Research},
  year={2025},
  publisher={INFORMS}
}

@article{Nesterov-1997-Self,
  title={Self-scaled barriers and interior-point methods for convex programming},
  author={Nesterov, Yu E and Todd, Michael J},
  journal={Mathematics of Operations Research},
  volume={22},
  number={1},
  pages={1--42},
  year={1997},
  publisher={INFORMS}
}

@article{Guler-1997-Hyperbolic,
  title={Hyperbolic polynomials and interior point methods for convex programming},
  author={G{\"u}ler, Osman},
  journal={Mathematics of Operations Research},
  volume={22},
  number={2},
  pages={350--377},
  year={1997},
  publisher={INFORMS}
}

@article{Dvurechensky-2025-Hessian,
  title={Hessian barrier algorithms for non-convex conic optimization},
  author={Dvurechensky, Pavel and Staudigl, Mathias},
  journal={Mathematical Programming},
  volume={209},
  number={1},
  pages={171--229},
  year={2025},
  publisher={Springer}
}

@article{Dai-2002-Convergence,
  title={Convergence properties of the BFGS algoritm},
  author={Dai, Yu-Hong},
  journal={SIAM Journal on Optimization},
  volume={13},
  number={3},
  pages={693--701},
  year={2002},
  publisher={SIAM}
}

@article{Mascarenhas-2004-bfgs,
  title={The BFGS method with exact line searches fails for non-convex objective functions},
  author={Mascarenhas, Walter F},
  journal={Mathematical Programming},
  volume={99},
  number={1},
  pages={49},
  year={2004},
  publisher={Springer Nature BV}
}

@article{Polyak-2009-Regularized,
  title={Regularized Newton method for unconstrained convex optimization},
  author={Polyak, Roman A},
  journal={Mathematical Programming},
  volume={120},
  number={1},
  pages={125--145},
  year={2009},
  publisher={Springer}
}

@techreport{Griewank-1981-Modification,
  title={The modification of Newton’s method for unconstrained optimization by bounding cubic terms},
  author={Griewank, Andreas},
  year={1981},
  institution={Technical report NA/12}
}

@article{Levenberg-1944-Method,
  title={A method for the solution of certain non-linear problems in least squares},
  author={Levenberg, Kenneth},
  journal={Quarterly of Applied Mathematics},
  volume={2},
  number={2},
  pages={164--168},
  year={1944}
}

@article{Marquardt-1963-Algorithm,
  title={An algorithm for least-squares estimation of nonlinear parameters},
  author={Marquardt, Donald W},
  journal={Journal of the Society for Industrial and Applied Mathematics},
  volume={11},
  number={2},
  pages={431--441},
  year={1963},
  publisher={SIAM}
}

@article{Tikhonov-1963-Solution,
  title={Solution of incorrectly formulated problems and the regularization method.},
  author={Tikhonov, Andrei N},
  journal={Sov Dok},
  volume={4},
  pages={1035--1038},
  year={1963}
}

@misc{wang-kfac-pytorch,
  author       = {Wang, Chaoqi},
  title        = {KFAC-Pytorch: PyTorch implementation of K-FAC and E-KFAC},
  year         = {2019},
  howpublished = {\url{https://github.com/alecwangcq/KFAC-Pytorch}},
  note         = {Accessed: 2026-03-11}
}

@article{Loshchilov-2017-Decoupled,
  title={Decoupled weight decay regularization},
  author={Loshchilov, Ilya and Hutter, Frank},
  journal={arXiv preprint arXiv:1711.05101},
  year={2017}
}

\end{document}